\documentclass[11pt, leqno, twoside]{article}   
\usepackage[applemac]{inputenc}
\usepackage[T1]{fontenc}
\usepackage{mathtools, bm}
\usepackage{amsthm}
\usepackage{amssymb, bm}
\usepackage{bbm}
\usepackage[mathscr]{euscript}
\usepackage[margin=1in]{geometry}
\usepackage{systeme}
\usepackage[french]{babel}
\usepackage{chngcntr}
\usepackage{titling}
\usepackage{cancel}
\usepackage{stmaryrd}
\usepackage{graphicx}
\usepackage{siunitx}
\usepackage{fancyhdr, ifthen}
\usepackage[normalem]{ulem} 

\addto\captionsfrench{%
}

\newenvironment{poliabstract}[1]
  {\renewcommand{\abstractname}{#1}\begin{abstract}}
  {\end{abstract}}

\newtheoremstyle{theoremnoperiod}
  {\topsep}   
  {\topsep}   
  {\itshape}  
  {0pt}       
  {\bfseries} 
  {}          
  {5pt plus 1pt minus 1pt} 
  {}          

\newtheorem*{theorem*}{Théorème}
\newtheorem{theorem}{Théorème}[section]
\newtheorem{corollary}[theorem]{Corollaire}

\newtheorem{lemma}[theorem]{Lemme}

\newtheorem*{conjecture*}{Conjecture}

\newtheorem{proposition}[theorem]{Proposition}

\numberwithin{equation}{section}

\theoremstyle{theoremnoperiod}
\newtheorem*{thmnodot*}{Théorème}

\newtheorem{lemmanodot}[theorem]{Lemme}

\DeclareMathOperator{\li}{li}

\DeclareMathOperator{\N}{\mathbb{N}}

\DeclareMathOperator{\R}{\mathbb{R}}
\DeclareMathOperator{\Q}{\mathbb{Q}}
\DeclareMathOperator{\C}{\mathbb{C}}

\DeclareMathOperator{\sD}{\mathscr{D}}
\DeclareMathOperator{\sC}{\mathscr{C}}

\DeclareMathOperator{\sA}{\mathscr{A}}
\DeclareMathOperator{\sP}{\mathscr{P}}
\DeclareMathOperator{\sE}{\mathscr{E}}
\DeclareMathOperator{\sF}{\mathscr{F}}

\DeclareMathOperator{\sgn}{\rm{sgn}}

\DeclareMathOperator{\e}{\rm e}
\def\d{\,{\rm d}}
\def\1{{\bf 1}}
\DeclareMathOperator{\PP}{\mathbb P}

\DeclareMathOperator{\gL}{\mathfrak L}

\DeclareMathOperator{\sJ}{\mathscr J}
\DeclareMathOperator{\sG}{\mathscr G}
\DeclareMathOperator{\sU}{\mathscr U}
\DeclareMathOperator{\sH}{\mathscr H}

\DeclareMathOperator{\sR}{\mathscr R}
\DeclareMathOperator{\sQ}{\mathscr Q}

\DeclareMathOperator{\sK}{\mathscr K}

\DeclareMathOperator{\gc}{\mathfrak c}
\DeclareMathOperator{\gr}{\mathfrak r}
\DeclareMathOperator{\gs}{\mathfrak s}
\DeclareMathOperator{\gh}{\mathfrak h}

\DeclareMathOperator{\gf}{\mathfrak f}

\DeclareMathOperator{\gv}{\mathfrak v}

\DeclareMathOperator{\gp}{\mathfrak p}

\DeclareMathOperator{\gE}{\mathfrak E}

\DeclareMathOperator{\gI}{\mathfrak I}

\def\p{{{p}_{\nu}}}
\def\po{{{p}_{\omega}}}
\def\pO{{{p}_{\Omega}}}
\def\res{\mathop{\hbox{\rm R\'es}}\nolimits}

\renewcommand{\leq}{\leqslant}

\renewcommand{\geq}{\geqslant}

\newcommand{\CommaBin}{\mathbin{\raisebox{0.45ex}{\rm ,}}}
\usepackage{color}
\definecolor{vert}{rgb}{0,0.5,0}
\definecolor{violet}{rgb}{0.7,0.1,0.8}
\definecolor{orange}{rgb}{1,0.3,0}

\graphicspath{ {./images/} }
\title{Étude statistique du facteur premier médian, 4 : \goodbreak
somme des inverses}
\author{Jonathan Rotgé\thanks{Adresse e-mail : jonathan.rotge@etu.univ-amu.fr}\\ \\ {\it\small  Université d'Aix-Marseille,\, Institut de Mathématiques de Marseille CNRS UMR 7373,}\\ {\it\small 163 Avenue De Luminy, Case 907, 13288 Marseille Cedex 9, FRANCE}}
\date{}

\begin{document}
\maketitle
\thispagestyle{empty}

\selectlanguage{english}
\begin{poliabstract}{Abstract} 
We consider the sum of the reciprocals of the middle prime factor of an integer, defined according to multiplicity or not. We obtain an asymptotic expansion in the first case and an asymptotic formula involving an implicit parameter in the second. Both these results improve on previous estimates available in the literature.
 \end{poliabstract}

\selectlanguage{french}
\begin{poliabstract}{Résumé}
   Nous nous intéressons à la somme des inverses du facteur premier médian d'un entier, défini en tenant compte ou non de la multiplicité. Nous obtenons un développement asymptotique de cette quantité dans le premier cas et une formule asymptotique faisant intervenir un paramètre défini implicitement dans le second.
\end{poliabstract}
\bigskip


\section{Introduction et énoncé des résultats}


\subsection{Description et historique du problème}

Pour chaque entier naturel $n\geq 1$, posons
\[\omega(n):=\sum_{p|n}1,\qquad\Omega(n):=\sum_{p^k||n}k\]
Lorsque $n>1$, désignons par $\{q_j(n)\}_{1\leq j\leq\omega(n)}$ (resp. $\{Q_j(n)\}_{1\leq j\leq\Omega(n)}$) la suite croissante des facteurs premiers de $n$ comptés sans (resp. avec) multiplicité et notons $P^+(n)$ (resp. $P^-(n)$) son plus grand (resp. petit) facteur premier. Adoptons en outre les conventions $P^+(1)=1$ et $P^-(1)=+\infty$.\par
Depuis la fin des années 1970, de nombreux travaux ont porté sur la somme des inverses de facteurs premiers particuliers. En 1986, Erd\H os, Ivi\'c et Pomerance \cite{bib:erdos_ivic} obtiennent la formule asymptotique
\begin{equation}\label{eq:formule:somme_inverses_greatest_prime_factor}
	\sum_{n\leq x}\frac1{P^+(n)}=\bigg\{1+O\bigg(\frac{\sqrt{\log_2 x}}{\log x}\bigg)\bigg\}x\int_{2}^x\varrho\Big(\frac{\log x}{\log t}\Big)\frac{\d t}{t^2}\quad(x\geq 3),
\end{equation}
où $\varrho$ désigne la fonction de Dickman définie comme la solution continue sur $[0,\infty[$ de l'équation différentielle aux différences $u\varrho'(u)=-\varrho(u-1)$ pour $u>1$, avec la condition initiale $\varrho(u)=1$ pour $0\leq u\leq 1$. Dans les années suivantes, De Koninck \cite{bib:dekoninck:Pk} obtient une formule analogue à \eqref{eq:formule:somme_inverses_greatest_prime_factor} pour la somme des inverses du $k$-ième plus grand facteur premier des entiers. De Koninck et Galambos \cite{bib:dekoninck:random} démontrent ensuite un résultat semblable pour la somme des inverses d'un facteur premier choisi aléatoirement. À la conférence de théorie analytique des nombres d'Oberwolfach de 1984, Erd\H{o}s demanda à De Koninck s'il s'était intéressé à la somme des inverses du facteur premier médian
\[\p(n):=\begin{cases}q_{\lceil\omega(n)/2\rceil}(n)&{\rm si\;} \nu=\omega,\\ Q_{\lceil\Omega(n)/2\rceil}(n)&{\rm si\;} \nu=\Omega,\end{cases}\]
définie par
\begin{equation}\label{def:Snux}S_{\nu}(x):=\sum_{n\leq x}\frac1{\p(n)}\cdot\end{equation}
Près de trois décennies plus tard, De Koninck et Luca \cite{bib:dekoninck} apportent une réponse partielle dans le cas $\nu=\omega$ en démontrant que
\begin{equation}\label{eq:dekoninckluca}
	S_{\omega}(x)=\frac{x}{\log x}\e^{\{\sqrt 2+o(1)\}\sqrt{(\log_2 x)\log_3 x}}\quad (x\to\infty).
\end{equation}
Ici et dans la suite, $\log_k$ désigne la $k$-ième itérée du logarithme. Ce résultat a été précisé par Ouellet \cite{bib:ouellet} qui a obtenu 
\begin{equation}\label{eq:ouellet}
	\begin{aligned}
		\log \frac{S_\omega(x)}{x/\log x}=\sqrt{2(\log_2 x)\log_3 x}\bigg\{1+\frac{P_1(\log_4 x)}{\log_3 x}+\frac{P_2(\log_4 x)}{(\log_3 x)^2}+O\bigg(\frac1{\{\log_3 x\}^2}\bigg)\bigg\}\CommaBin
	\end{aligned}
\end{equation}
où les $P_j\ (j\in\{1,2\})$ sont les polynômes définis en \eqref{eq:coefficients:cnj:values} {\it infra}. \par 
Le cas $\nu=\Omega$ a été envisagé par Doyon et Ouellet dans \cite{bib:doyon}, qui établissent la formule asymptotique
\begin{equation}\label{eq:form_asymp_SOmega_ouellet}
	S_{\Omega}(x)=\frac{{\gc_1}x}{\sqrt{\log x}}\Big\{1+O\Big(\frac1{\log_2 x}\Big)\Big\}\quad (x\geq 3),
\end{equation}
où la constante $\gc_1$ est définie en \eqref{def:c1} {\it infra}. \par
À ce stade, deux questions sont donc en suspens: (a) déterminer un terme d'erreur optimal pour la formule asymptotique \eqref{eq:form_asymp_SOmega_ouellet}; (b) obtenir une véritable formule asymptotique pour $S_\omega(x)$.
\par 
Il est clair qu'une connaissance précise des lois locales \begin{equation}
\label{loisloc}
M_{\nu}(x,p):=|\{n\leq x:\p(n)=p\}|\quad (2\leq p\leq x),
\end{equation}
permettrait de résoudre ces deux problèmes.
Dans \cite{bib:papier2}, nous avons obtenu des estimations satisfaisantes lorsque $\varepsilon<\beta_p:=(\log_2 p)/\log_2 x<1-\varepsilon\ (\varepsilon>0)$. Cependant, il s'avère que la somme $S_\nu(x)$ est dominée par la contribution de valeurs de $\p(n)$ notablement plus petites que la borne inférieure de ce domaine: $\log p\approx \sqrt{\log_2x}$ si $\nu=\omega$ et $p\ll 1$ si $\nu=\Omega$. En effet, nous verrons en \eqref{eq:SOmstardevAsRes:2} que, dans le second cas, nous avons
\[S_{\Omega}(x)\sim\frac{x}{\sqrt{\log x}}\sum_{p}\frac{f(p)}{p^2}\CommaBin \]
avec $f(p)\ll(\log p)^{3/2}$.
Comme l'ont notamment noté McNew, Pollack et Singha Roy \cite{bib:pollack}, les comportements asymptotiques de $S_\omega(x)$ et $S_\Omega(x)$ diffèrent significativement. Nous avons
 \begin{equation}\label{eq:comp:SO;So}
 	\frac{S_{\Omega}(x)}{S_\omega(x)}=(\log x)^{1/2+o(1)}\quad (x\to\infty).
 \end{equation}
Par ailleurs, de \cite[th. 1.1]{bib:papier2} découle directement la relation
\begin{equation}\label{eq:equivalent:rapport:loislocales}
	\frac{M_\Omega(x,p)}{M_\omega(x,p)}\sim \frac{C(\beta_p)(\log x)^{1/2-(2\sqrt{\beta_p(1-\beta_p)}-3\beta_p/2)}}{\sqrt{\log_2 x}}\quad(\varepsilon>0,\, \varepsilon\leq\beta_p\leq \tfrac15-\varepsilon),
\end{equation}
pour une fonction $C(v)\asymp1$ explicite. \par
Ces phénomènes peuvent s'expliquer par la structure particulière des entiers ayant un facteur premier médian $p$ petit relativement à la taille de $x$. En effet, une approche naturelle révèle (voir {\it e.g.} \cite{bib:pollack},\,\cite{bib:papier2}) que les lois locales dépendent directement des deux quantités
\begin{equation}\label{def:phinuklambda}\Phi_{\nu,k}(x,p):=\sum_{\substack{b\leq x\\ P^-(b)> p\\ \nu(b)=k}}1,\quad \lambda_{\nu}(p,k):=\sum_{\substack{P^+(a)\leq p\\ \nu(a)=k}}\frac1a\quad(x,p\geq 2,\,k\geq 1)\CommaBin\end{equation}
portant respectivement sur les entiers criblés et sur les entiers friables. \par
Alors que la première somme possède un comportement asymptotique analogue pour les deux valeurs $\nu=\omega$ et $\nu=\Omega$ (voir {\it e.g.} \cite[th. 6 \& th.7]{bib:alladi}), la seconde présente un changement radical de comportement dans le cas $\nu=\Omega$ lorsque le rapport $k/\log_2 p$ dépasse $2$ (voir \cite[cor. 3.4, 3.5 \& 3.6]{bib:papier2}). Cette différence de comportement s'explique notamment par l'accumulation des petits facteurs premiers dans la décomposition des entiers friables (voir \cite[cor. 1.7]{bib:bret}).\par\smallskip
Notre premier objectif consiste à établir, par la méthode du col, une formule asymptotique pour $S_{\omega}(x)$ avec terme d'erreur effectif. L'estimation obtenue dépend d'un paramètre implicite, caractéristique de la méthode. Comme cela est génériquement le cas, l'insertion de ce paramètre fournit une évaluation explicite dont nous pouvons déduire un développement asymptotique pour le membre de gauche de \eqref{eq:ouellet}.\par
Notre second objectif est de préciser \eqref{eq:form_asymp_SOmega_ouellet}.. Au Théorème \ref{thm:somme_inverse:avec_mult}, nous obtenons un développement asymptotique pour $S_\Omega(x)$ dont la troncature à l'ordre $1$ fournit une formule asymptotique avec terme d'erreur relatif optimal $\asymp 1/(\log x)^{1/6}$, améliorant ainsi significativement \eqref{eq:form_asymp_SOmega_ouellet}. Il est à noter que le développement asymptotique \eqref{eq:somme_inverse:avec_mult} {\it infra} pour $S_\Omega(x)$ est relatif à des puissances de $1/\log x$ alors que l'évaluation obtenue dans \cite{bib:papier1} pour
$$T_\Omega(x):=\sum_{n\leqslant x}\log p_{\Omega,m}(n)$$
met en évidence un développement selon les puissances de $1/\log_2x$. Cette disparité est due au fait que les plages de valeurs dominantes pour $p_{\Omega,m}$ sont très éloignées : $p\leqslant \psi(x)$ dès que $\psi(x)\to\infty$ pour $S_\Omega(x)$, $\log p=(\log x)^{(\sqrt{5}-1)/2+o(1)}$ dans le cas de $T_\Omega(x)$.
\par
Mentionnons enfin que, dans les travaux \cite{bib:doyon} et \cite{bib:ouellet}, les auteurs considèrent plus généralement le facteur premier $\alpha$-positionné défini, pour $0<\alpha<1$, par 
\[p_\nu^{(\alpha)}(n):=\begin{cases}q_{\lceil\alpha\omega(n)\rceil}(n)&{\rm si\;} \nu=\omega,\\Q_{\lceil\alpha\Omega(n)\rceil}(n)&{\rm si\;} \nu=\Omega.\end{cases}\]
Nous avons choisi de restreindre l'étude au cas $\alpha=\tfrac12$ afin de préserver la clarté d'exposition et d'éviter la multiplication de détails techniques sans véritable intérêt théorique. Nos méthodes sont adaptables à l'obtention d'estimations de même précision dans le cas général. Nous n'approfondirons pas cette question dans la suite de cet article.\par


\subsection{Notations et résultats}

\par Notons $\PP:=\{\gp_j:j\geq 1\}$ l'ensemble des nombres premiers. Dans toute la suite, les lettres $p$ et $q$ désignent également des nombres premiers. Introduisons les notations
\begin{equation}\label{def:Hnu;H0}
	\begin{gathered}
		F_\nu(y,z):=\begin{cases}
			\displaystyle\prod_{q\leq y}\Big(1+\frac{z}{q-1}\Big)\quad (y>0,\,z\in\C,\,\nu=\omega),\\
			\displaystyle\prod_{q\leq y,\, q\neq z}\Big(1-\frac{z}{q}\Big)^{-1}\quad (y>0,\,z\in\C,\,\nu=\Omega),
		\end{cases}\\
	\Psi_x(v):=\frac{\log_2 x}{v^2}-\frac{F_\omega'(\e^v,v)}{F_\omega(\e^v,v)}\quad(x\geq 3,\, v>0),\quad \varrho_x:=\inf\{v>0:\Psi_x(v)\leq 0\}\quad(x\geq 3),
	\end{gathered}
\end{equation}
où l'on a posé $F_\omega'(y,z)=\partial F_\omega(y,z)/\partial z$. Nous verrons au Lemme \ref{l:existence_varrho_x} que la quantité $\varrho_x$ est finie. Un développement asymptotique de $\varrho_x$ est obtenu en \eqref{eq:approximation:rho} {\it infra}. Nous verrons dans la suite que la quantité $S_\omega(x)$ peut être approchée par l'intégrale $S_\omega^{**}(x)$ définie en \eqref{def:alpha;j;eta} {\it infra} et dont le paramètre $\varrho_x$ constitue une approximation précise du col.\par 


\begin{theorem}\label{thm:somme_inverse:sans_multiplicite}
	Nous avons
	\begin{equation}\label{eq:somme_inverse:sans_multiplicite}
		S_\omega(x)=\frac{xF_\omega(\e^{\varrho_x},\varrho_x)\e^{-\varrho_x}}{(\log x)^{1-1/\varrho_x}\sqrt{2\log_2 x}}\bigg\{1+O\bigg(\frac{1}{\log_3 x}\bigg)\bigg\}\quad\big(x\geq 10^7\big).
	\end{equation}
\end{theorem}



\begin{corollary}\label{cor:explicit_formula}
	Il existe une suite de polynômes $\{P_{j}\}_{j\in\N}$ telle que, pour tout $J\in\N$,
	\begin{equation}\label{eq:formule_explicite:S}
		\log \frac{S_\omega(x)}{x/\log x}=\sqrt{2(\log_2 x)\log_3 x}\Bigg\{\sum_{0\leq j\leq J}\frac{P_j(\log_4 x)}{(\log_3 x)^{j}}+O_{J}\bigg(\bigg\{\frac{\log_4 x}{\log_3 x}\bigg\}^{J+1}\bigg)\Bigg\}\quad(x\geq 10^7).
	\end{equation}
	De plus, pour tout $j\in\N$, $P_j$ est de degré $j$ et de coefficient dominant $3^j(2j)!/(1-2j)4^j(j!)^2$. En particulier, nous avons
	\begin{equation}\label{eq:coefficients:cnj:values}
		\begin{gathered}
			P_0(X)=1,\quad P_1(X):=-\tfrac32X+\tfrac{3{\gL}}{2}-1,\quad P_2(X)=-\tfrac98X^2+\big(\tfrac{9{\gL}}{4}+1\big)X-\tfrac{9{\gL}^2}{8}-{\gL}+2,\\
			P_3(X)=-\tfrac{27}{16}X^3+\big(\tfrac{81{\gL}}{16}+\tfrac{25}{8}\big)X^2-\big(\tfrac{81{\gL}^2}{16}+\tfrac{25{\gL}}4-2\big)X+\tfrac{27{\gL}^3}{16}+\tfrac{25{\gL}^2}{8}-2{\gL}+4+\tfrac{2\pi^2}3\CommaBin
		\end{gathered}
	\end{equation}
	où l'on a posé $\gL=\log 2$.
\end{corollary}
\noindent{\it Remarque.} Les polynômes $P_j\ (j\geq 0)$ sont explicités en \eqref{def:polynomesPj:expl}.\medskip


La formule implicite \eqref{eq:somme_inverse:sans_multiplicite} fournit également des informations sur le comportement local de $S_\omega(x)$.


\begin{corollary}\label{cor:local_behaviour}
	Soient $0<\varepsilon<1$ et $K_0>0$ fixés. Nous avons 
	\begin{equation}\label{eq:asymptotics:Sx}
		S_\omega(x^{1+h})=\frac{x^hS_\omega(x)}{1+h}\bigg\{1+O_{K_0,\varepsilon}\bigg(\frac{1}{\log_3 x}\bigg)\bigg\}\CommaBin
	\end{equation}
	uniformément pour $x\geq 10^{7/\varepsilon}$ et $0\leqslant h\leqslant \exp\big(K_0\sqrt{\log_2 x}/(\log_3 x)^{3/2}\big)$.\par
	\smallskip
	En particulier, nous avons $S_\omega(x^{1+h})\sim {x^h}S_\omega(x)$ lorsque $x\to\infty$ et $h\to0^+$.
\end{corollary}


Rappelons la définition de $F_\nu(y,z)$ en \eqref{def:Hnu;H0} et posons 
\begin{equation}\label{def:g;c}
	\begin{gathered}
		\sF_\nu(z):=\begin{cases}
			\displaystyle\prod_q\Big(1+\frac{z}{q-1}\Big)\Big(1-\frac1q\Big)^z\quad &(z\in\C,\,\nu=\omega),\\
			\displaystyle\prod_q\Big(1-\frac{z}q\Big)^{-1}\Big(1-\frac1q\Big)^z\quad &(|z|<2,\,\nu=\Omega),
		\end{cases}\\
		\gc_j:=\frac{(1+\gp_j)\sF_\Omega(1/{\gp_j})}{\Gamma(1/{\gp_j})}\sum_{p\geq {\gp_j}}\frac{F_\Omega(p,{\gp_j})}{p(p-1/{\gp_j})F_\Omega(p,1/{\gp_j})}\quad(j\geq 1).
	\end{gathered}
\end{equation}
Les produits eulériens $\sF_\nu(z)$ interviennent naturellement dans les estimations de type Sathe-Selberg (voir {\it e.g.} \cite[ch. II.6]{tenenbaum_livre}). Notons que les coefficients $\gc_j$ sont bien définis puisque
\[\frac{F_\Omega(p,{\gp_j})}{p(p-1/{\gp_j})F_\Omega(p,1/{\gp_j})}\asymp_j\frac{(\log p)^{\gp_j-1/{\gp_j}}}{p^2}\cdot\]
\par
Notons que $\sgn F_\Omega(p,{\gp_j})=(-1)^{j-1}$ et $F_\Omega(p,{1/\gp_j})>0$ pour tout $p\geqslant \gp_j$, et donc $\sgn \gc_j=(-1)^{j-1}$ pour tout $j\geqslant 1$. En particulier, nous avons
\begin{equation}\label{def:c1}
	\gc_1=\frac9{4\sqrt\pi}\sum_p\frac{1}{p(p-1/2)}\prod_{3\leq q\leq p}\frac{q-1/2}{q-2}\prod_q\frac{\sqrt{q(q-1)}}{q-1/2}\approx1{,}380486,\quad\gc_2\approx -1{,}311133.
\end{equation}


\begin{theorem}\label{thm:somme_inverse:avec_mult}
	Pour tout $J\geq1$, nous avons
	\begin{equation}\label{eq:somme_inverse:avec_mult}
		S_\Omega(x)=\sum_{1\leq j\leq J}\frac{{\gc_j}x}{(\log x)^{1-1/{\gp_j}}}+O_{J}\bigg(\frac{x}{(\log x)^{1-1/{\gp_{J+1}}}}\bigg)\quad(x\geq 3).	
	\end{equation}
\end{theorem}


Le corollaire suivant est une conséquence immédiate du Théorème \ref{thm:somme_inverse:avec_mult} appliqué avec $J=1$.


\begin{corollary}
	Nous avons l'estimation optimale
\begin{equation}
\label{faSom}
S_\Omega(x)=\frac{{\gc_1}x}{\sqrt{\log x}}\bigg\{1+O\bigg(\frac1{(\log x)^{1/6}}\bigg)\bigg\}\quad(x\geq 3).
\end{equation}
\end{corollary}


La formule \eqref{faSom} est notablement plus simple que \eqref{eq:somme_inverse:sans_multiplicite}. Comme nous le verrons dans la suite, cette disparité est due à la nature du col relevant de chaque situation.\par
La section \ref{sec:study_of_varrho} est consacrée à l'obtention d'un développement asymptotique pour le paramètre implicite $\varrho_x$ défini en \eqref{def:Hnu;H0}. À la section \ref{sec:main_contribution}, nous déterminons les valeurs de $\p(n)$ et $\nu(n)$ dominant la somme $S_\nu(x)$. Les sections \ref{section:reduction_Sstar} à \ref{sec:proof;th1} sont dévolues à la démonstration du Théorème~\ref{thm:somme_inverse:sans_multiplicite}; la section \ref{sec:proof;th4} à celle du Théorème~\ref{thm:somme_inverse:avec_mult}. Enfin, la preuve des corollaires est donnée à la section \ref{sec:explicit_formula}.



\section{Preuve du Théorème \ref{thm:somme_inverse:sans_multiplicite}: étude du paramètre implicite $\varrho_x$}\label{sec:study_of_varrho}

Rappelons les définitions de $\Psi_x$ et $\varrho_x$ en \eqref{def:Hnu;H0} et posons
\begin{equation}\label{def:Li;Ei;xi}	
	\xi=\xi_x:=\log_2 x\quad(x\geq 3).
\end{equation}


\begin{lemma}\label{l:existence_varrho_x}
	La fonction $\Psi_x$ admet exactement un changement de signe sur $\R^+$. De plus, nous avons
	\begin{equation}\label{eq:estimation_rho}
		|\Psi_x(\varrho_x)|\leq\e^{-\varrho_x}\quad(x\geq 3).
	\end{equation}
\end{lemma}
\begin{proof}
	D'après les définitions \eqref{def:Hnu;H0}, nous avons
	\begin{equation}\label{eq:reecr:psi_x}
		\Psi_x(v)=\frac{\xi}{v^2}-\sum_{q\leq\e^v}\frac1{q-1+v}\quad(v>0).
	\end{equation}
	Remarquons d'emblée que, pour $x$ assez grand,
	\begin{equation*}\label{eq:intervalle_rho}
		\varrho_x\in\sR_x:=\Bigg[\sqrt{\frac{\xi}{\log\xi}},\sqrt{\frac{5\xi}{\log\xi}}\Bigg]\quad(x\geq 16).
	\end{equation*}
	En effet, nous avons d'une part, pour $0<v<\sqrt{\xi/\log\xi}$,
	\[\Psi_x(v)\geq \log\xi-\sum_{q\leq\e^v}\frac1{q-1}\geq \tfrac12\log\xi+O(\log_2 \xi)>0\quad(x\to\infty),\]
	et d'autre part, pour $v>\sqrt{5\xi/\log\xi}$,
	\[\Psi_x(v)\leq\tfrac15\log\xi-\frac{\pi(v)}{2v}-\tfrac12\sum_{v<q\leq\e^v}\frac{1}{q}\leq-\tfrac1{20}\log\xi+O(\log_2\xi)<0\quad(x\to\infty).\]
	\par
	Posons 
	\[\sQ_{j}:=\;]\log \gp_j,\log \gp_{j+1}[\quad (j\in\N^*).\] 
	La fonction $\Psi_x$ est dérivable sur $\sQ_{j}$. Pour $j\in\N^*$, $v\in \sQ_j\cap\sR_x$, une intégration par parties fournit
	\begin{equation}\label{eq:derivative;Psi}
		\begin{aligned}
			\Psi_x'(v)+\frac{2\xi}{v^3}&=\sum_{q\leq\e^v}\frac1{(q+v-1)^2}=\int_{2^-}^{\e^v}\frac{\d\pi(t)}{(t+v-1)^2}=2\int_2^{\e^v}\frac{\pi(t)\d t}{(t+v-1)^3}+O\Big(\frac1{v^2}\Big)\\
			&\ll\frac1{v^2}+\int_2^{\e^v}\frac{t\d t}{(t+v-1)^3\log t}\ll\frac1{v^2}\int_2^v\frac{\d t}{\log t}+\frac1{\log v}\int_v^{\e^v}\frac{\d t}{t^2}\ll\frac1{v\log v}\cdot
		\end{aligned}
	\end{equation}
	Il suit, pour $x$ assez grand,
	\[\Psi'_x(v)\leq-\frac{2(\log\xi)^{3/2}}{5\sqrt{5\xi}}+O\bigg(\frac{1}{\sqrt{\xi\log\xi}}\bigg)\quad(v\in\sQ_j\cap\sR_x,\,j\in\N^*).\]
	$\Psi_x$ est donc strictement décroissante sur $\sR_x$ avec des sauts aux seuls points $\{\log\gp_j\}_{j\geq 1}$ de sorte que $\Psi_x(\varrho_x)\neq0$ implique $\varrho_x\in\{\log\gp_j:j\geqslant 1\}$. Or, $|\Psi_x(\log\gp_j)|=1/(\gp_j+\log\gp_j-1)\leq1/\gp_j\leq\e^{-\varrho_x}$. Cela conclut la démonstration.
\end{proof}


Posons
\begin{equation*}\label{def:I(v)}
	\gI(v):=\int_2^{\e^v}\frac{\d t}{(t-1+v)\log t}\quad(v>0).
\end{equation*} 


\begin{lemma}
	Nous avons l'estimation
	\begin{equation}\label{eq:approx_rho_TNP}
		\frac{\xi}{\varrho_x^2}=\gI(\varrho_x)+O\big(\e^{-\sqrt{\log \varrho_x}}\big)\quad (x\geq 3).
	\end{equation}
\end{lemma}
\begin{proof}
	Nous avons
	\begin{equation*}\label{eq:eval:derlog:Frho}
		\frac{F_\omega'(\e^{v},v)}{F_\omega(\e^{v},v)}=\int_2^{\e^{v}}\frac{\d\pi(t)}{t-1+v}=\gI(v)+\int_{2}^{\e^v}\frac{\d\{\pi(t)-\li(t)\}}{t-1+v}\quad (v>0).
	\end{equation*}
	Un calcul analogue à celui mené en \eqref{eq:derivative;Psi} fournit alors que l'intégrale est $\ll\e^{-\sqrt{\log v}}$.
\end{proof}



\subsection{Estimation de $\gI(v)$}\label{subsec:shift}

Commençons par un lemme technique fournissant un développement asymptotique pour deux familles d'intégrales.


\begin{lemma}\label{l:an_integral}
	Soient $m\geq 2$, $n\geq 0$ et $J\geq 1$. Lorsque $v\to\infty$, nous avons les estimations
	\begin{gather}
		\label{eq:an_integral}\int_{v}^{\infty}\frac{\d t}{t^m\log t}=\frac1{v^{m-1}}\Bigg\{\sum_{1\leq j\leq J}\frac{(-1)^{j+1}(j-1)!}{(\{m-1\}\log v)^j}+O_{J}\bigg(\frac{1}{\{m\log v\}^{J+1}}\bigg)\Bigg\}\CommaBin\\
		\label{eq:a_second_integral}\int_2^v\frac{t^n\d t}{\log t}=\sum_{1\leq j\leq J}\frac{v^{n+1}(j-1)!}{(\{n+1\}\log v)^j}-\frac{2^{n+1}}{n+1}\Big\{\frac1{\log 2}+O\Big(\frac1{n+1}\Big)\Big\}+O_{J}\bigg(\frac{v^{n+1}}{\{(n+1)\log v\}^{J+1}}\bigg)\cdot
	\end{gather}
\end{lemma}
\begin{proof}
	Le résultat découle d'intégrations par parties successives et des majorations
	\[\int_{v}^\infty\frac{\d t}{t^{m}(\log t)^{J+1}}\leq\frac{1}{(\log v)^{J+1}}\int_{v}^\infty\frac{\d t}{t^m},\quad \int_2^v\frac{t^n\d t}{(\log t)^{J+1}}\ll_J\frac{1}{(\log v)^{J+1}}\int_{\sqrt v}^vt^n\d t.\qedhere\]
\end{proof}


Posons
\begin{equation}\label{def:alpha_m}
	\alpha_j:=2(j-1)!\sum_{n\geq 1}\frac{(-1)^{n-1}}{n^j}=2(j-1)!(1-2^{1-j})\zeta(j)\quad(j\geq 1),
\end{equation}
où $\zeta$ désigne la fonction zêta de Riemann, avec la convention que le membre de droite vaut $2\log 2$ pour $j=1$.
\par Notons que 
\begin{equation}\label{eq:exp:alpha2j}
	\alpha_{2j}=\frac{(2^{2j-1}-1)\pi^{2j}|B_{2j}|}{j}\quad (j\geqslant 1),
\end{equation}
où $B_{2j}$ désigne le $(2j)$-ième nombre de Bernoulli.


\begin{lemma}\label{prop:AsEx:sigma}
	Soit $J\in\N$. Nous avons
	\begin{equation}\label{eq:sum:recriprocals:shifted:primes}
		\gI(v)=\log\Big(\frac{v}{\log v}\Big)+\sum_{1\leq j\leq J}\frac{\alpha_{2j}}{(\log v)^{2j}}+O_{J}\bigg(\frac{1}{\{\log v\}^{2J+2}}\bigg)\quad(v\geq 2).
	\end{equation}
\end{lemma}
\begin{proof}
	Posant
	\[\gI_1(v):=\int_{2}^{v}\frac{\d t}{(t+v)\log t},\quad\gI_2(v):=\int_{v}^{\e^{v}}\frac{\d t}{(t+v)\log t}\quad (v\geq 2),\]
	nous pouvons écrire
	\begin{equation}\label{eq:sum:minus1:comp}
		\gI(v)=\int_{2}^{\e^v}\frac{\d t}{(t+v)\log t}+\int_2^{\e^v}\frac{\d t}{(t+v)(t-1+v)\log t}=\gI_1(v)+\gI_2(v)+O\Big(\frac{\log v}{v}\Big)\cdot
	\end{equation}
	Par ailleurs, nous avons
	\begin{align*}
		\gI_2(v)&=\int_{v}^{\e^{v}}\frac{\d t}{t\log t}\sum_{n\geq 0}\Big(\frac{-v}{t}\Big)^n=\log\Big(\frac{v}{\log v}\Big)+\sum_{n\geq 1}(-v)^n\int_{v}^{\e^{v}}\frac{\d t}{t^{n+1}\log t}\CommaBin
	\end{align*}
	d'où, d'après \eqref{eq:an_integral},
	\begin{equation}\label{eq:estimation:Ivarrho}
		\begin{aligned}
			\gI_2(v)
			&=\log\Big(\frac{v}{\log v}\Big)+\tfrac12\sum_{1\leq j\leq 2J+1}\frac{(-1)^j\alpha_j}{(\log v)^j}+O_{J}\bigg(\frac1{\{\log v\}^{2J+2}}\bigg)\cdot
		\end{aligned}
	\end{equation}
	De manière analogue, 
	\begin{equation}\label{eq:est:I2rho:inter}
		\gI_1(v)=\frac{1}{v}\int_2^{v}\frac{\d t}{\log t}\sum_{n\geq 0}\Big(\frac{-t}{v}\Big)^n=\sum_{n\geq 0}\frac{(-1)^n}{v^{n+1}}\int_2^{v}\frac{t^n\d t}{\log t}\cdot
	\end{equation}
	En insérant \eqref{eq:a_second_integral} dans \eqref{eq:est:I2rho:inter}, il vient
	\begin{equation}\label{eq:estimation:I2varrho}
		\begin{aligned}
			\gI_1(v)&=\sum_{n\geq 0}\sum_{1\leq j\leq 2J+1}\frac{(-1)^n(j-1)!}{(\{n+1\}\log v)^j}+\sum_{n\geq 1}\frac{(-2)^{n}}{nv^n}\bigg\{\frac{1}{\log 2}+O\Big(\frac1n\Big)\bigg\}+O_{J}\bigg(\frac{1}{\{\log v\}^{2J+2}}\bigg)\\
			&=\tfrac12\sum_{1\leq j\leq 2J+1}\frac{\alpha_j}{(\log v)^j}+O_{J}\bigg(\frac{1}{\{\log v\}^{2J+2}}\bigg)\cdot
		\end{aligned}
	\end{equation}
	La formule \eqref{eq:sum:recriprocals:shifted:primes} est obtenue en injectant les estimations \eqref{eq:estimation:Ivarrho} et \eqref{eq:estimation:I2varrho} dans \eqref{eq:sum:minus1:comp}.
\end{proof}



\subsection{Développement asymptotique de $\varrho_x$}

Rappelons la définition de $\xi$ en \eqref{def:Li;Ei;xi} et posons
\[\mu_x:=\sqrt{\frac{2\xi}{\log\xi}}\quad(x\geq 16).\]


\begin{lemma}\label{l:equation_nu}
	Pour tout $x\geq 5$, l'équation 
	\begin{equation}\label{def:equation_nu}
		v^2\gI(v)=\xi
	\end{equation}
	admet une unique solution $\nu=\nu_x$ sur $[1,+\infty[$ et l'on a $\nu_x\sim\mu_x$ lorsque $x\to\infty$. De plus, 
	\begin{equation}
		\label{eq:relation_rho_nu}\varrho_x=\nu_x\big\{1+O\big({\e^{-\frac12\sqrt{\log\xi}}}\big)\big\}\quad(x\geq 16).\\
	\end{equation}
\end{lemma}
\begin{proof} 
	Posons $\sJ(v):=v^2\gI(v)\ (v>0)$. Nous avons
	$$\sJ'(v)=\int_2^{\e^v}\frac{2v(t-1)+v^2}{(t-1+v)^2\log t}\d t+\frac{v\e^v}{\e^v-1+v}> 0\quad(v\geqslant 1).$$
	Comme $\sJ(\infty)=\infty$, il s'ensuit que $\sJ([1,\infty])=[\sJ(1),\infty[=[-\log_22,\infty[$. Nous obtenons bien ainsi l'existence et l'unicité de $\nu_x$ pour $x\geqslant 5>\e^{1/\log 2}$. 
	\par 
	Posons $H_x(v):=\gI(v)-\xi/v^2\ (v>0)$, de sorte que $H_x(\nu_x)=0,$ et, par \eqref{eq:approx_rho_TNP}, $H_x(\varrho_x)\ll\e^{-\sqrt{\log \varrho_x}}$. Comme 
	\[H'_x(v)=\frac{2\xi}{v^3}+O\Big(\frac1{v}\Big)\CommaBin\] 
	 le théorème des accroissements finis implique bien \eqref{eq:relation_rho_nu} en remarquant que $\nu_x\asymp \sqrt{\xi/\log \xi}\asymp\varrho_x$, d'après \eqref{eq:sum:recriprocals:shifted:primes}.
\end{proof}


Nous dirons qu'une fonction réelle $\gf$ admet un développement asymptotique au voisinage de $+\infty$ s'il existe une suite réelle $\{c_j\}_{j\geq 0}$ et une suite de fonctions $\{\varphi_j\}_{j\geq 0}$ vérifiant $\varphi_{j+1}(v)=o(\varphi_j(v))$ $ (j\geq 0,\,v\to\infty)$ et telles que, pour tout $J\geq 0$, on ait 
\begin{equation*}\label{eq:dev_asymp_f}
	\gf(v)=\sum_{0\leq j\leq J}c_j\varphi_j(v)+O_{J}\big(\varphi_{J+1}(v)\big)\quad(v\to\infty).
\end{equation*}
Nous écrivons alors 
\begin{equation}\label{eq:set_of_dev_asymp_f}
	\gf(v)\approx\sum_{j\geq 0}c_j\varphi_j(v)\quad(v\to\infty).
\end{equation}
Notons qu'il n'est pas nécessaire que la série \eqref{eq:set_of_dev_asymp_f} soit convergente (voir {\it e.g.} \cite[$\S 1.4$, $\S 1.5$]{bib:debruijn}). Dans la suite, nous notons $[v^n]\gf(v)$ le coefficient de $v^n$ dans le développement en série entière de $\gf(v)$.


\begin{proposition}\label{prop:dev_asymp:rho}
	Il existe une suite de polynômes $\{R_{j}\}_{j\in\N}$ telle que, pour tout $J\in\N$,
	\begin{equation}\label{eq:approximation:rho}
		\varrho_x=\sqrt{\frac{2\xi}{\log\xi}}\Bigg\{\sum_{0\leq j\leq J}\frac{R_j(\log_2\xi)}{(\log\xi)^{j}}+O\bigg(\bigg\{\frac{\log_2\xi}{\log\xi}\bigg\}^{J+1}\bigg)\Bigg\}\quad\big(x\geq 10^7\big).
	\end{equation}	
	Pour tout $j\in\N$, $R_j$ est de degré $j$, de coefficient dominant $3^j(2j)!/4^j(j!)^2$ et, notant $\gL=\log 2$, nous avons
	\begin{equation*}\label{eq:coefficients:rho}
		\begin{gathered}
			R_0(X)=1,\quad R_1(X)=\tfrac32X-\tfrac{3{\gL}}2,\quad R_2(X)=\tfrac{27}8X^2-\big(\tfrac{27{\gL}}4+\tfrac52\big)X+\tfrac{27{\gL}^2}{8}+\tfrac{5{\gL}}{2},\\
			R_3(X)=\tfrac{135}{16}X^3-\big(\tfrac{405{\gL}}{16}+14\big)X^2+\big(\tfrac{405{\gL}^2}{16}+28{\gL}+\tfrac{11}2\big)X-\tfrac{135{\gL}^3}{16}-14{\gL}^2-\tfrac{11{\gL}}{2}-\tfrac{2\pi^2}{3}\cdot
		\end{gathered}
	\end{equation*}
\end{proposition}
\noindent{\it Remarque}. Les polynômes $R_j\ (j\geq 0)$ sont explicités en \eqref{def:zlj;Rj}. La contribution progressive des facteurs $\alpha_{2m}$ dans l'expression des coefficients des polynômes $R_j$ ne permet d'obtenir une forme close simple que pour le facteur dominant. Notons cependant que $[X^\ell]R_j(X)=P_{j,\ell}(\log 2)$ pour certains $P_{j,\ell}\in\R_{j-\ell}[X]\ (j\geq 0,\,0\leq \ell\leq j)$.
\begin{proof}
	D'après \eqref{eq:relation_rho_nu}, il nous suffit d'expliciter un développement asymptotique de $\nu_x$. Rappelons la définition des coefficients $\alpha_j$ en \eqref{def:alpha_m}. D'après \eqref{eq:sum:recriprocals:shifted:primes}, nous avons
	\begin{equation}\label{eq:serexp:I}
		\gI(v)\approx\log\Big(\frac{v}{\log v}\Big)+\sum_{j\geq 1}\frac{\alpha_{2j}}{(\log v)^{2j}}\cdot
	\end{equation}
	Notons $\nu_x=\mu_x\e^{\gh_x}$. Remarquons d'emblée que $\gh_x=o(1)$ $(x\to\infty)$ puisque $\nu_x\sim\mu_x$ d'après le Lemme \ref{l:equation_nu}. Posons
	\begin{equation*}\label{def:g;sigma;f}
		\begin{gathered}
			\lambda_x(z):=\gI(\mu_x\e^z)-\gI(\mu_x)\quad(|z|<1),\quad\sigma_x:=\frac{1}{\log\xi}\quad(x\geq 16),\\
			\tau_x:=\frac{\log(\frac12\log \xi)}{\log\xi}\quad\big(x\geq 10^7\big),\quad f_x(z):=\frac{z}{\e^{-2z}-1-2\sigma_x \lambda_x(z)}\quad(0<|z|<1).
		\end{gathered}
	\end{equation*}
	Notons que $f_x$ possède une singularité apparente en $z=0$ et admet donc un prolongement analytique au voisinage de l'origine. Posant
	\begin{equation}\label{def:ck}
		c_{x,k}:=\frac{1}{k!}\bigg[\frac{\d^{k-1}f_x(z)^k}{\d z^{k-1}}\bigg]_{z=0}\quad(k\geq 1),\quad w_x:=2\sigma_x \gI(\mu_x)-1\quad(x\geq 16),
	\end{equation}
	l'équation \eqref{def:equation_nu} équivaut à
	\begin{equation}\label{eq:defV}
		w_x=\frac{\gh_x}{f_x(\gh_x)}\cdot
	\end{equation}
	Par ailleurs, l'estimation \eqref{eq:serexp:I} implique $\gI(\mu_x)=\tfrac12\log\xi+O(\log_2\xi)$ de sorte que $w_x=O(\sigma_x\log_2\xi)=o(1)$. D'après le théorème d'inversion de Lagrange (voir {\it e.g.} \cite[$\S$2.2]{bib:debruijn}), l'équation \eqref{eq:defV} admet une unique solution donnée par
	\begin{equation}\label{eq:expression;V}
		\gh_x=\sum_{k\geq 1}c_{x,k}w_x^k.
	\end{equation}
	Rappelons les expressions des réels $\alpha_{2j}$ obtenues en \eqref{eq:exp:alpha2j} et introduisons pour $|z|<1$
	\begin{equation}\label{def:series_formelles}
		\begin{gathered}
			\Lambda(z;s,t):=z-\log\Big(1+\frac{2zs}{1-t}\Big)+\sum_{j\geq 1}\frac{\alpha_{2j}(2s)^{2j}}{(1-t)^{2j}}\bigg\{\bigg(1+\frac{2zs}{1-t}\bigg)^{-2j}-1\bigg\},\\W(s,t):=-3t-2s\log(1-t)+\sum_{j\geq 1}\frac{\alpha_{2j}(2s)^{2j+1}}{(1-t)^{2j}},\\ F(z;s,t):=\frac{z}{\e^{-2z}-1-2s\Lambda(z;s,t)},\\
			H(s,t):=\sum_{k\geq 1}\frac{\{W(s,t)\}^k}{k}[z^{k-1}]F(z;s,t)^k,\quad E(s,t):=\e^{H(s,t)}\cdot
		\end{gathered}
	\end{equation}
	Les séries impliquées dans les définitions \eqref{def:series_formelles} sont formelles. Les opérations de composition et d'exponentiation étant compatibles avec les développements asymptotiques à tout ordre fini, les troncatures de ces séries fournissent, après la substitution $s=\sigma_x$ et $t=\tau_x$, des développements asymptotiques des quantités $w_x$ et $\gh_x$ données par \eqref{def:ck} et \eqref{eq:expression;V}. En particulier, puisque
	\[\log\mu_x=\frac{1-\tau_x}{2\sigma_x},\quad \log_2\mu_x=\frac{\tau_x}{\sigma_x}+\log(1-\tau_x),\]
	nous obtenons, d'après la relation \eqref{eq:serexp:I},
	\begin{equation}\label{eq:serexp:w}
		w_x\approx W(\sigma_x,\tau_x).
	\end{equation}
	Au vu des définitions \eqref{def:series_formelles} et de la relation \eqref{eq:serexp:w}, nous déduisons également $\e^{\gh_x}\approx E(\sigma_x,\tau_x)$. Puisque, de plus, $\tau_x=\sigma_x\log(\tfrac12\log\xi)$, nous obtenons
	\begin{equation}\label{eq:devasehx}
		\begin{aligned}
			\e^{\gh_x}&\approx\sum_{k,\ell\geq 0} \{[s^{k}t^\ell]E(s,t)\}\sigma_x^k\tau_x^\ell\\
			&\approx \sum_{k,\ell\geq 0}\{[s^{k}t^\ell]E(s,t)\}\sigma_x^{k+\ell}(\log_2\xi-\log 2)^\ell\\
			&\approx\sum_{j\geq 0}\sigma_x^jR_j(\log_2\xi),
		\end{aligned}
	\end{equation}
	où l'on a posé
	\begin{equation}\label{def:zlj;Rj}R_j(X):=\sum_{0\leq \ell\leq j}\{[s^{j-\ell}t^\ell]E(s,t)\}(X-\log 2)^\ell,\end{equation}
	ce qui établit \eqref{eq:approximation:rho}. Remarquons enfin que le coefficient dominant de $R_j$ est donné par
	\[[X^j]R_j(X)=[t^j]E(0,t)=[t^j]\frac1{\sqrt{1-3t}}=\big(\tfrac{3}{4}\big)^j\binom{2j}{j}.\]
	Cela achève la démonstration. 
\end{proof}




\section{Domaines de contribution principale}\label{sec:main_contribution}
Cette section est dévolue à l'explicitation de plages de valeurs pour $\p(n)$ et $\nu(n)$ dominant la somme $S_\nu(x)$. Les deux valeurs $\nu=\omega$ et $\nu=\Omega$ sont considérées.\par
Écrivons tout entier $n\geq 2$ sous la forme $n=a{p_{\omega}(n)}^\ell b=A{\pO(n)}B$ avec $P^+(a)< p<P^-(b)$, $P^+(A)\leq p\leq P^-(B)$, $\omega(b)-\omega(a)=\1_{\{2|\omega(n)\}}$ et $\Omega(B)-\Omega(A)=\1_{\{2|\Omega(n)\}}$. La décomposition est clairement unique dans le cas $\nu=\omega$; elle l'est aussi lorsque $\nu=\Omega$  puisque les conditions indiquées impliquent $\Omega(A)=\lfloor{(\Omega(n)-1)/2}\rfloor$. Rappelons la définition de $\Phi_{\nu,k}(x,y)$ en \eqref{def:phinuklambda} et remarquons que
\begin{equation}\label{somme_complete}
	\begin{aligned}
		S_\omega(x)&=\sum_{p\leq x}\frac1p\sum_{k\leq\frac{\log x}{\log 4}}\sum_{\substack{a\leq x/p\\P^+(a)< p\\\omega(a)=k}}\sum_{\ell\leq\frac{\log x}{\log 2}}\bigg\{\Phi_{\omega,k}\Big(\frac x{ap^\ell},p\Big)+\Phi_{\omega,k+1}\Big(\frac x{ap^\ell},p\Big)\bigg\}=:S_{\omega,\iota}(x)+S_{\omega,\pi}(x),\\
		S_\Omega(x)&=\sum_{p\leq x}\frac{1}{p}\sum_{\substack{A\leq x/p\\ P^+(A)\leq p}}\sum_{\substack{B\leq x/Ap\\ P^-(B)\geq p}}\big(\1_{\{\Omega(A)=\Omega(B)\}}+\1_{\{\Omega(A)=\Omega(B)-1\}}\big)=:S_{\Omega,\iota}(x)+S_{\Omega,\pi}(x).
	\end{aligned}
\end{equation}
Commençons par évaluer la contribution $S_{\nu,\iota}(x)$ correspondant aux entiers $n\leq x$ pour lesquels $\nu(n)$ est impair. La première étape consiste à dégager un domaine de contribution principale pour les variables $p$ et $k$. \par
Rappelons la définition de $\varrho_x$ en \eqref{def:Hnu;H0}, celle de $\xi:=\log_2x$ en \eqref{def:Li;Ei;xi}, posons
\begin{equation}\label{def:intervalles:P;K}
	\begin{gathered}
		q_{1,x}:=\e^{\varrho_x/(\log\varrho_x)^2}=\e^{\{1+o(1)\}\sqrt{32\xi/(\log\xi)^5}},\quad q_{2,x}:=\e^{\varrho_x(\log\varrho_x)^2}=\e^{\{1+o(1)\}\sqrt{\xi(\log\xi)^{3}/8}},\\
		\sP_{\omega,x}:=[q_{1,x},q_{2,x}],\quad \sK_{\omega,x}:=\bigg[1,\frac{\xi}{\log\xi}\bigg],\quad\sP_{\Omega,x}:=[2,\log x],\quad\sK_{\Omega,x}:=\big[1,\tfrac32\xi\big],\\
		\sA_{\nu,x}:=\{n\leq x:\p(n)\in\sP_{\nu,x},\,\Omega(n)\leq 10\xi,\, \nu(n)=2k+1,\,k\in\sK_{\nu,x}\}\quad \big(x\geq 16\big)
	\end{gathered}
\end{equation}
et notons
\begin{equation}\label{def:Siotastar}
	S_{\nu,\iota}^*(x):=\sum_{n\in\sA_{\nu,x}}\frac1{\p(n)}\quad(x\geq 16),
\end{equation}
la contribution à $S_{\nu,\iota}(x)$ des entiers $n\leq x$ appartenant au domaine $\sA_{\nu,x}$. Les deux résultats suivants permettent de restreindre l'étude de $S_{\nu,\iota}(x)$ à celle de $S_{\nu,\iota}^*(x)$.


\begin{proposition}
	Il existe une constante $c_0>\tfrac32$ telle que
	\begin{equation}\label{eq:est:S;Sstar}
		S_{\omega,\iota}(x)=S_{\omega,\iota}^*(x)+O\Big(\frac{x}{\log x}\exp\Big\{{\sqrt{2(\log_2 x)\log_3 x}\Big(1-\frac{c_0\log_4 x}{\log_3 x}\Big)}\Big\}\Big)\quad\big(x\geq  10^7\big).
	\end{equation}
\end{proposition}
\begin{proof}
	Rappelons la définition de $\xi$ en \eqref{def:Li;Ei;xi}. D'après \cite[lemme 2.1]{bib:papier1}, nous avons
	\[\sum_{\substack{n\leq x\\ \Omega(n)>10\xi}}\frac{1}{\po(n)}\ll\frac{x}{(\log x)^5},\quad\sum_{\substack{n\leq x\\ \po(n)>\log x}}\frac{1}{\po(n)}\leq\sum_{n\leq x}\frac1{\log x}\leq\frac{x}{\log x}\cdot\]
	Posons
	\[E(x):=\sum_{\substack{n\in[1,x]\smallsetminus\sA_{\omega,x}\\ \omega(n)\text{ impair}\\\Omega(n)\leq10\xi\\ \po(n)\leq\log x}}\frac{1}{\po(n)}=\Bigg(\sum_{\substack{n\leq x\\\omega(n)\text{ impair}\\ \omega(a)\notin\sK_{\omega,x}\\ \Omega(n)\leq10\xi\\ \po(n)\leq\log x}}+\sum_{\substack{n\leq x\\\omega(n)\text{ impair}\\ \omega(a)\in\sK_{\omega,x}\\ \po(n)\notin\sP_{\omega,x}\\ \po(n)\leq\log x}}\Bigg)\frac{1}{\po(n)}=:E_1(x)+E_2(x)\quad (x\geq  16),\]
	de sorte que
	\begin{equation}\label{eq:S;Sstar;E}
		S_{\omega,\iota}(x)=S_{\omega,\iota}^*(x)+E(x)+O\Big(\frac x{\log x}\Big)\cdot
	\end{equation}
	Nous évaluons $E(x)$ en suivant la méthode employée dans \cite[$\S$4]{bib:ouellet}. L'inégalité de Hardy-Ramanujan permet de majorer $E_1(x)$ par
	\begin{equation}\label{eq:majo:E1:inter:1}
		\begin{aligned}
			\sum_{p\leq \log x}\frac1p\sum_{\substack{k\notin\sK_{\omega,x}\\ 1\leq k\leq 10\xi}}\sum_{\substack{a\leq x\\ P^+(a)<p\\ \omega(a)=k}}\sum_{\ell\leq10\xi}\sum_{\substack{b\leq x/ap^\ell\\\omega(b)=k}}1\ll\frac{x}{\log x}\sum_{p\leq\log x}\frac1{p^2}\sum_{\substack{k\notin\sK_{\omega,x}\\ 1\leq k\leq 10\xi}}\frac{\xi^{k-1}}{(k-1)!}\sum_{\substack{P^+(a)< p\\ \omega(a)=k}}\frac1a\cdot
		\end{aligned}
	\end{equation}
	Par ailleurs, le théorème de Mertens fournit la majoration
	\begin{equation}\label{eq:majo:sum:reciprocals:smooth}
		\sum_{\substack{P^+(a)< p\\ \omega(a)=k}}\frac1a\ll\frac{1}{k!}\bigg(\sum_{q\leq p}\frac{1}{q-1}\bigg)^{k}\ll\frac{(\log_2 p+M_0)^{k}}{k!}\CommaBin
	\end{equation}
	où $M_0$ désigne la constante de Meissel-Mertens. Regroupant les majorations \eqref{eq:majo:E1:inter:1} et \eqref{eq:majo:sum:reciprocals:smooth}, nous obtenons
	\begin{equation}\label{eq:majo:Sstar:inter:1}
		\begin{aligned}
			E_1(x)&\ll\frac{x}{\xi\log x}\sum_{p\leq\log x}\frac1{p^2}\sum_{\substack{k\notin\sK_{\omega,x}\\ k\leq 10\xi}}\frac{(\{\log_2 p+M_0\}\xi)^k}{k!(k-1)!}\\
			&\ll\frac{x}{\xi\log x}\sum_{p\leq \log x}\frac1{p^2}\sum_{\xi/\log\xi<k\leq 10\xi}\bigg(\frac{\e\sqrt{\{\log_2 p+M_0\}\xi}}{k}\bigg)^{2k}\CommaBin
		\end{aligned}
	\end{equation}
	d'après la formule de Stirling. Puisque la fonction $v\mapsto(\e\lambda/v)^v$ atteint son maximum en $v=\lambda$, nous déduisons de \eqref{eq:majo:Sstar:inter:1} que
	\begin{equation}\label{eq:majo:E11}
		E_1(x)\ll\frac{x}{\log x}\sum_{p\leq \log x}\frac{1}{p^2}\bigg(\frac{\{\log\xi\}^{3/2}}{\sqrt{\xi}}\bigg)^{2\xi/\log\xi}\ll\frac{x\xi}{\log x}\cdot
	\end{equation}\par
	Il reste à majorer $E_2(x)$. Rappelons les définitions de $q_{1,x}$ et $q_{2,x}$ en \eqref{def:intervalles:P;K}. L'estimation \eqref{eq:approximation:rho} avec $J=0$ implique
	\begin{equation*}\label{eq:varrho:asymptotics}
		\varrho_x=\sqrt{\frac{2\xi}{\log\xi}}\bigg\{1+O\bigg(\frac{\log_2\xi}{\log\xi}\bigg)\bigg\}\quad(x\geq10^7).
	\end{equation*}
	Par suite, 
	\[\log_2 q_{1,x}=\tfrac12\big\{{\log}\xi-5\log_2 \xi+5\log 2+o(1)\big\},\quad\log_2 q_{2,x}=\tfrac12\big\{{\log}\xi+3\log_2 \xi-3\log 2+o(1)\big\}.\]
	Le terme d'erreur $E_2(x)$ est donc
	\begin{equation}\label{eq:majo:E22}
		\begin{aligned}
			\ll\frac{x}{\xi\log x}\sum_{p\notin\sP_{\omega,x}}\frac1{p^2}\sum_{k\in\sK_{\omega,x}}\bigg(\frac{\e\sqrt{\{\log_2 p+M_0\}\xi}}{k}\bigg)^{2k}\ll\frac{x}{\log x}\sum_{p\notin \sP_{\omega,x}}\frac{\e^{2\sqrt{(\log_2 p+M_0)\xi}}}{p^2}\cdot
		\end{aligned}
	\end{equation}
	Or, nous avons d'une part
	\begin{equation}\label{eq:majo:E2:smallp}
		\sum_{p< q_{1,x}}\frac{\e^{2\sqrt{(\log_2 p+M_0)\xi}}}{p^2}\ll\exp\bigg(\sqrt{2\xi\log\xi}\bigg\{1-\frac{5\log_2 \xi}{2\log\xi}+O\bigg(\frac{1}{\log\xi}\bigg)\bigg\}\bigg)\CommaBin
	\end{equation}
	et, d'autre part,
	\begin{equation}\label{eq:majo:E2:bigp}
		\sum_{q_{2,x}< p\leq \log x}\frac{\e^{2\sqrt{(\log_2 p+M_0)\xi}}}{p^2}\ll\frac{\e^{\sqrt{2\xi\log\xi}}}{\e^{2\sqrt{\xi\log\xi}}}=o(1).
	\end{equation}
	En regroupant les estimations \eqref{eq:S;Sstar;E}, \eqref{eq:majo:E11}, \eqref{eq:majo:E22}, \eqref{eq:majo:E2:smallp} et \eqref{eq:majo:E2:bigp} nous obtenons  \eqref{eq:est:S;Sstar}.
\end{proof}



\begin{proposition}
	Nous avons
	\begin{equation}\label{eq:est:S;Sstar:Omega}
		S_{\Omega,\iota}(x)=S_{\Omega,\iota}^*(x)+O\Big(\frac{x}{\log x}\Big)\quad(x\geq 16).
	\end{equation}
\end{proposition}
\begin{proof}
	Nous avons d'abord, trivialement,
	\[\sum_{\substack{n\leq x\\ \pO(n)>\log x}}\frac1{\pO(n)}\leq \frac{x}{\log x}\cdot\]
	Ensuite, d'après \cite[lemme 2.1]{bib:papier1} appliqué avec $k=\lfloor 3\xi\rfloor$, nous avons
	\begin{equation}\label{eq:majo:bigOmega}
		\sum_{\substack{n\leq x\\ \Omega(n)>3\xi+1}}\frac{1}{\pO(n)}\leq\sum_{p\leq x}\frac1p\sum_{\substack{d\leq x/p\\ \Omega(d)\geq\lfloor3\xi\rfloor}}1\ll\frac{x\log_2 x}{(\log x)^{\log 8 -1}}\ll\frac{x}{\log x}\cdot\qedhere
	\end{equation}
\end{proof}




\section{Preuve du Théorème \ref{thm:somme_inverse:sans_multiplicite}: préparation technique}
\label{section:reduction_Sstar}

Rappelons la définition de $F_\nu(y,z)$ en \eqref{def:Hnu;H0}, celles de $\sF_\nu(z)$ en \eqref{def:g;c} et de $\xi$ en \eqref{def:Li;Ei;xi} et posons
\begin{equation}\label{def:g}
	\begin{gathered}
		g_\Omega(y,z):=\frac{\sF_\Omega(z)}{(1-z/y)F_\Omega(y,z)},\quad g_\omega(y,z):=\frac{\sF_\omega(z)}{F_\omega(y,z)}\quad(y\geq 2,\, |z|<2),\\
		\gr_{x,t}:=\frac{t-1}{\xi}\quad(x\geq 3,\,t\in\R).
	\end{gathered}
\end{equation}
Nous ferons usage du résultat suivant dû à Alladi. Rappelons la définition $\Phi_{\nu,k}(x,y)$ en \eqref{def:phinuklambda}.


\begin{lemmanodot}{\bf (\cite[th. 6]{bib:alladi}).}\label{thm:Phikxy:eval} 
	Soit $\varepsilon>0$. Sous les conditions $3\leq y\leq\e^{(\log x)^{2/5}}$ et $0< \gr_{x,k}\leq 2-\varepsilon$, nous avons
	\begin{equation}\label{eq:Phikxy:eval}
		\Phi_{\omega,k}(x,y)=\frac{xg_\omega(y,\gr_{x,k})\xi^{k-1}}{(\log x)\Gamma(1+\gr_{x,k})(k-1)!}\bigg\{1+O\bigg(\frac{k\{\log_2 y\}^2}{\xi^2}\bigg)\bigg\}\quad(x\geq 3).
	\end{equation}
\end{lemmanodot}


Posons, pour $k\geq 1$, $p\geq 3$,
\begin{equation}\label{def:alpha;j;eta}
	\begin{gathered}
		\sU_x:=\bigg[\frac{\varrho_x}{(\log\varrho_x)^2},\varrho_x(\log\varrho_x)^2\bigg],\quad \sD_{x}^{\pm}:=\bigg]-\infty\pm \frac{i\xi}{\varrho_x},\pm \frac{i\xi}{\varrho_x}\bigg],\quad\sC_x:=\bigg\{\frac{\xi{\e^{i\vartheta}}}{\varrho_x}:|\vartheta|\leq\tfrac{\pi}{2}\bigg\},\\
		S_{\omega,\iota}^{**}(x):=\frac{x}{2\pi i\xi\log x}\int_{\sH_x}\int_{\sU_x}\frac{F_\omega(\e^u,\xi/s)\e^{s-u}\d u\d s}{u}\quad (x\geq 3),
	\end{gathered}
\end{equation}
où $\sH_x$ est un contour de Hankel constitué des deux demi-droites $\sD_x^+$, $\sD_x^-$ et du demi-cercle $\sC_x$.


\begin{proposition}
	Nous avons
	\begin{equation}\label{eq:reduction:Sstar}
		S_{\omega,\iota}^*(x)=S_{\omega,\iota}^{**}(x)\Big\{1+O\Big(\frac1{\log_3 x}\Big)\Big\}\quad(x\geq 16).
	\end{equation}
\end{proposition}
\begin{proof}
	Pour $p\in\sP_{\omega,x}$, $k\in\sK_{\omega,x}$, nous avons $p\leq\log x$ et, puisque $\Omega(n)\leq10\log_2 x$, il vient $ap^\ell\leq(\log x)^{10\log_2 x+1}$, et donc
	\[p<\log x<\e^{\{\log(x/ap^\ell)\}^{2/5}}\quad (1\leq\ell\leq \Omega(n)),\]
	pour $x$ suffisamment grand. Nous pouvons donc évaluer $\Phi_{\omega,k}(x,p)$ par le Lemme \ref{thm:Phikxy:eval} lorsque $(p,k)\in \sP_{\omega,x}\times\sK_{\omega,x}$. L'estimation \eqref{eq:Phikxy:eval} fournit ainsi
	\[S_{\omega,\iota}^*(x)=\bigg\{1+O\bigg(\frac{\log_3 x}{\log_2 x}\bigg)\bigg\}x\sum_{p\in\sP_{\omega,x}}\sum_{k\in\sK_{\omega,x}}\sum_{\substack{a\leq x/p\\ P^+(a)< p\\ \omega(a)=k}}\sum_{1\leq \ell\leq10\xi}\frac{g_\omega(p,\gr_{x/ap^\ell,k})\{\log_2 (x/ap^\ell)\}^{k-1}}{ap^{\ell+1}\log(x/ap^\ell)\Gamma(1+\gr_{x/ap^\ell,k})(k-1)!}\cdot\]
	De plus, sous les conditions $\Omega(n)\leq10\log_2 x$ et $p\leq \log x$, nous avons
	\begin{gather*}
		\log\frac{x}{ap^\ell}=\bigg\{1+O\bigg(\frac{\{\log_2 x\}^2}{\log x}\bigg)\bigg\}\log x,\quad\log_2\Big(\frac{x}{ap^\ell}\Big)=\Big\{1+O\Big(\frac{\log_2 x}{\log x}\Big)\Big\}\log_2 x,\\
		\gr_{x/ap^\ell,k}=\gr_{x,k}\Big\{1+O\Big(\frac{\log_2 x}{\log x}\Big)\Big\}\quad(1\leq\ell\leq \Omega(n)).
	\end{gather*}
	Enfin, le théorème des accroissements finis implique
	\begin{equation*}\label{eq:AF:g;Gamma}
		g_\omega(p,\gr_{x,k})=1+O\Big(\frac1{\log_3 x}\Big),\quad \Gamma(1+\gr_{x,k})=1+O\Big(\frac1{\log_3 x}\Big)\quad(x\geq 16,\, k\in\sK_{\omega,x}).
	\end{equation*}
	Il vient
	\begin{equation}\label{eq:eval:Siota:inter:1}
		S_{\omega,\iota}^*(x)=\frac{x}{\log x}\sum_{p\in\sP_{\omega,x}}\frac1{p^2}\sum_{k\in\sK_{\omega,x}}\frac{\xi^{k-1}\lambda_\omega(p,k)}{(k-1)!}\Big\{1+O\Big(\frac{1}{\log\xi}\Big)\Big\}\quad(x\geq 16),
	\end{equation}
	puisque
	\[\sum_{1\leq \ell\leq 10\xi}\frac{1}{p^{\ell+1}}=\frac{1}{p^2}\Big\{1+O\Big(\frac{1}{\log\xi}\Big)\Big\}\quad(p\in\sP_{\omega,x}).\]
	Par ailleurs, d'après \eqref{eq:majo:sum:reciprocals:smooth} et la formule de Stirling, nous avons la majoration
	\begin{align*}
		\sum_{\xi/\log\xi<k\leq\pi(p)}\frac{\xi^{k-1}\lambda_\omega(p,k)}{(k-1)!}&\ll\frac1\xi\sum_{\xi/\log\xi<k\leq\pi(p)}\bigg(\frac{\e\sqrt{\xi\{\log_2 p+M_0\}}}{k}\bigg)^{2k}\\
		&\ll\frac{\pi(p)}{\xi}\bigg(\frac{\{\log\xi\}^{3/2}}{\sqrt{\xi}}\bigg)^{2\xi/\log\xi}\ll\frac{p}{\log p}\cdot
	\end{align*}
	Il suit
	\begin{equation}\label{eq:eval:Siota:error:1}
		\frac{x}{\log x}\sum_{p\in\sP_{\omega,x}}\frac1{p^2}\sum_{\xi/\log\xi<k\leq\pi(p)}\frac{\xi^{k-1}\lambda_\omega(p,k)}{(k-1)!}\ll\frac{x}{\log x}\cdot
	\end{equation}
	Ce terme d'erreur est pleinement acceptable au vu de \eqref{eq:dekoninckluca}. En regroupant l'estimation \eqref{eq:eval:Siota:inter:1} et la majoration \eqref{eq:eval:Siota:error:1}, nous obtenons
	\[S_{\omega,\iota}^*(x)=\frac{x}{\log x}\sum_{p\in\sP_{\omega,x}}\frac1{p^2}\sum_{k\leq\pi(p)}\frac{\xi^{k-1}\lambda_\omega(p,k)}{(k-1)!}\Big\{1+O\Big(\frac{1}{\log\xi}\Big)\Big\}\quad(x\geq 16).\]
	La formule de Hankel (voir {\it e.g.} \cite[th. II.0.17]{tenenbaum_livre}) fournit alors
	\begin{equation}\label{eq:eval:Siotasstar:hankel:sum}
		S_{\omega,\iota}^*(x)=\frac{x\{1+O(1/\log\xi)\}}{2\pi i\xi\log x}\sum_{p\in\sP_{\omega,x}}\frac{1}{p^2}\int_{\sH_x}\bigg\{\sum_{k\leq\pi(p)}\Big(\frac{\xi}{s}\Big)^k\lambda_\omega(p,k)\bigg\}{\e^s}{\d}s.
	\end{equation}
	Puisque 
	\[k>\pi(p)\Rightarrow \lambda_\omega(p,k)=0\quad(p\geq 2),\]
	nous déduisons de \eqref{eq:eval:Siotasstar:hankel:sum} et de l'identité
	\[\sum_{k\geq 0}\lambda_\omega(p,k)z^k=F_\omega(p,z)\quad(z\in\C,\,p\geq 2),\]
	l'estimation
	\begin{equation}\label{eq:est:Somegaiotastarint}S^*_{\omega,\iota}(x)
	=\frac{x\{1+O(1/\log\xi)\}}{2\pi i \xi\log x}\int_{\sP_{\omega,x}}\frac{\d t}{t^2\log t}\int_{\sH_x}F_\omega\Big(t,\frac{\xi}{s}\Big)\e^s\d s,\end{equation}
	où nous avons utilisé le théorème des nombres premiers. La formule \eqref{eq:reduction:Sstar} résulte alors du changement de variables $u=\log t$ dans \eqref{eq:est:Somegaiotastarint}.
\end{proof}


Pour tout chemin $\gamma$ de $\C^*$, définissons
\begin{equation}\label{def:gJ;IE}
	I_{\gamma}(x):=\frac1{2\pi i}\int_{\gamma}\int_{\sU_x}\frac{F_\omega(\e^u,\xi/s)\e^{s-u}\d u\d s}{u}\quad(x\geq 3),
\end{equation}
de sorte que, au vu des définitions \eqref{def:alpha;j;eta}, nous avons
\begin{equation}\label{eq:formule:Siotasstar}
	S_{\omega,\iota}^{**}(x)=\frac{x}{\xi\log x}\big\{I_{\sC_x}(x)+I_{\sD_x^+}(x)+I_{\sD_x^-}(x)\big\}.
\end{equation}
Posons encore
\begin{equation}\label{def:Thetax;C;+C-}
	\Theta:=\big[-\pi/3,\pi/3\big],\quad\sC_x^+:=\{s\in\sC_x:\arg s\in\Theta\},\quad\sC_x^-:=\sC_x\smallsetminus\sC_x^+\quad(x\geq 3).
\end{equation}
Nous verrons que la contribution principale à $S_{\omega,\iota}^{**}(x)$ est obtenue le long de $\sC_x^+$. Nous évaluons $I_{\sC_x^+}$ à la section \ref{sec:mainterm}, puis, à la section \ref{sec:errorterm}, nous obtenons une majoration des termes d'erreur $I_{\sC_x^-}$, $I_{\sD_x^+}$ et $I_{\sD_x^-}$.



\section{Évaluation de $I_{\sC_x^+}$}\label{sec:mainterm}

Soit $\varepsilon>0$. Posons
\begin{equation}\label{def:gen_incomplete_gamma}
	\Gamma(z;a,b):=\int_{a}^{b}t^{z-1}\e^{-t}\d t\quad(a,b>0,\, |{\arg}\,z|\leq\tfrac\pi2-\varepsilon).
\end{equation}


\begin{lemma}\label{l:gen_incomplete_gamma} Soient $\varepsilon>0$ et $\psi:\R^+\to\R^+$ une fonction vérifiant $\lim_{t\to\infty}\psi(t)=+\infty$. Sous les conditions $r\geqslant 1$, $|\vartheta|\leq\tfrac{\pi}2-\varepsilon$, $z=r\e^{i\vartheta}$, nous avons uniformément
	\begin{equation}\label{eq:gen_incomplete_gamma}
		\Gamma\Big(z;\frac{r}{\psi(r)},r\psi(r)\Big)=\sqrt{2\pi}z^{z-1/2}\e^{-z}\Big\{1+O_{\varepsilon}\Big(\frac1{r}\Big)\Big\}\cdot
	\end{equation}
\end{lemma}
\begin{proof}
	Le changement de variables $t=z\e^s$ dans \eqref{def:gen_incomplete_gamma} permet d'écrire
	\begin{equation}\label{def:Gammastar}
		\Gamma^*(z):=z^{-z}\e^z\Gamma\Big(z;\frac{r}{\psi(r)},r\psi(r)\Big)=\int_{\log\{r/z\psi(r)\}}^{\log\{r\psi(r)/z\}}\e^{-z(\e^s-s-1)}\d s,
	\end{equation}
	où le logarithme complexe est pris en détermination principale. Puisque $z=r{\e^{i\vartheta}}$, nous avons
	\[\Gamma^*(z)=\int_{-\log\psi(r)-i\vartheta}^{\log\psi(r)-i\vartheta}\e^{-z(\e^s-s-1)}\d s.\]
	Comme $s\mapsto\e^{-z(\e^s-s-1)}$ est une fonction entière, le théorème intégral de Cauchy implique
	\begin{equation}\label{eq:cutting:Gammastar:cauchy}
		\Gamma^*(z)=\bigg\{\int_{-\log\psi(r)-i\vartheta}^{-\log\psi(r)}+\int_{-\log\psi(r)}^{\log\psi(r)}+\int_{\log\psi(r)}^{\log\psi(r)-i\vartheta}\bigg\}\e^{-z(\e^s-s-1)}\d s=:G_1(z)+G_2(z)+G_3(z).
	\end{equation}
	\par Pour $s=\log\psi(r)-i\tau\ (0\leq\tau\leq\vartheta)$, nous avons
	\begin{align*}
		\big|{\e^{-z(\e^s-s-1)}}\big|&=\exp\big(\Re\big\{-r{\e^{i\vartheta}}\big(\psi(r){\e^{-i\tau}}-\log\psi(r)+i\tau-1\big)\big\}\big)\\
		&=\e^{-r\psi(r)\cos(\vartheta-\tau)+r\cos\vartheta\{1+\log\psi(r)\}+r\tau\sin\vartheta}\ll\e^{-r\{\psi(r)\cos\vartheta-\log\psi(r)-\vartheta-1\}},
	\end{align*}
	uniformément en $\tau$. Il vient
	\begin{equation}\label{eq:majo:error2}
		G_3(z)\ll\vartheta\e^{-r\{\psi(r)\cos\vartheta-\log\psi(r)-\vartheta-1\}}\ll\frac{1}{r^{3/2}}\cdot
	\end{equation}
	\par De manière analogue, nous avons
	\begin{equation}\label{eq:majo:error1}
		G_1(z)\ll\vartheta\e^{-r\{\log\psi(r)-1\}}\ll\frac{1}{r^{3/2}}\cdot
	\end{equation}
	En insérant \eqref{eq:majo:error2} et \eqref{eq:majo:error1} dans \eqref{eq:cutting:Gammastar:cauchy}, nous obtenons
	\begin{equation}\label{eq:eval:Gammastar:inter}
		\Gamma^*(z)=G_2(z)+O\Big(\frac1{r^{3/2}}\Big)\cdot
	\end{equation}\par
	Désignons par $G_2^+(z)$ la contribution à $G_2(z)$ de l'intervalle $[-r^{-1/3},r^{-1/3}]$ et posons $G_2^-(z):=G_2(z)-G_2^+(z)$. Puisque
	\[\e^v-v-1=\tfrac12v^2+\tfrac16v^3+O(v^4)\quad(v\ll 1),\]
	nous obtenons, par symétrie de l'intervalle d'intégration,
	\begin{equation}\label{eq:eval:G+}
		G_2^+(z)=\int_{-r^{-1/3}}^{r^{-1/3}}\e^{-zv^2/2}\big\{1+O\big(r v^4+r^2v^6\big)\big\}\d v=\sqrt{\frac{2\pi}{z}}\Big\{1+O\Big(\frac{1}{r}\Big)\Big\}\cdot
	\end{equation}
	Par ailleurs,
	\begin{equation}\label{eq:eval:G-}
		G_2^-(z)\ll\int_{|v|\geq r^{-1/3}}\e^{-r(\e^v-v-1)\cos\vartheta}\d v\ll \e^{-c_\varepsilon r^{1/3}}\ll_\varepsilon\frac1{r^{3/2}}\CommaBin
	\end{equation}
	de sorte que la formule \eqref{eq:gen_incomplete_gamma} découle des estimations \eqref{def:Gammastar}, \eqref{eq:eval:Gammastar:inter}, \eqref{eq:eval:G+} et \eqref{eq:eval:G-}.\qedhere
\end{proof}


Rappelons la définition de $\varrho_x$ en \eqref{def:Hnu;H0}.


\begin{lemma} Nous avons
	\begin{equation}\label{eq:formule:IC+}
		I_{\sC_x^+}(x)=\frac{F_\omega({\e^{\varrho_x}},\varrho_x)\e^{\xi/\varrho_x-\varrho_x}\sqrt{\xi}}{\sqrt 2\varrho_x}\Big\{1+O\Big(\frac1{\log_3 x}\Big)\Big\}\quad (x\geq 16).
	\end{equation}
\end{lemma}
\begin{proof}
	D'après la définition \eqref{def:gJ;IE}, nous avons
	\[I_{\sC_x^+}(x)=\frac{1}{2\pi i}\int_{\sC_x^+}\int_{\sU_x}\frac{F_\omega(\e^u,\xi/s)\e^{s-u}\d u\d s}{u}\cdot\]
	Rappelons la définition de $\Theta$ en \eqref{def:Thetax;C;+C-}, posons $z_\vartheta=z_{x,\vartheta}:=\varrho_x\e^{-i\vartheta}\ (\vartheta\in\R)$ et notons que $\sC_x^+=\{\xi/z_\vartheta:\vartheta\in\Theta\}$, de sorte que
	\[I_{\sC_x^+}(x)=\frac{\xi}{2\pi\varrho_x}\int_{\Theta}\int_{\sU_x}\frac{F_\omega(\e^u,z_\vartheta)\e^{\xi/z_\vartheta+i\vartheta-u}\d u\d\vartheta}{u}\cdot\]
	Avec le changement de variables $u=\varrho_x(1+t)$, nous obtenons
	\begin{equation}\label{eq:Ic+:inter:2}
		I_{\sC_x^+}(x)=\frac{\xi\e^{-\varrho_x}}{2\pi\varrho_x}\int_{\Theta}\e^{\xi/z_\vartheta+i\vartheta}\d\vartheta\int_{1/(\log\varrho_x)^2-1}^{(\log\varrho_x)^2-1}\frac{F_\omega(\e^{\varrho_x\{1+t\}},z_\vartheta)\e^{-\varrho_x t}\d t}{1+t}\cdot
	\end{equation}
	Posons
	\begin{equation*}\label{def:varphi}
		\varphi(\vartheta,v):=\log F_\omega(\e^v,z_\vartheta)=\sum_{q<\e^v}\log\Big(1+\frac{z_\vartheta}{q-1}\Big)\quad(\vartheta\in\Theta,\,v\geq0).
	\end{equation*}
	Pour $0\leq t\leq (\log\varrho_x)^2-1$ et $x$ assez grand, nous avons 
	\begin{equation}\label{eq:est:diff:Avartheta}
		\begin{aligned}
			\varphi(\vartheta,\varrho_x\{1+t\})-\varphi(\vartheta,\varrho_x)&=\sum_{\e^{\varrho_x}\leq q<\e^{\varrho_x(1+t)}}\log\Big(1+\frac{z_\vartheta}{q-1}\Big)\\
			&=z_\vartheta\sum_{\e^{\varrho_x}\leq q<\e^{\varrho_x(1+t)}}\frac1{q-1}+O\bigg(\varrho_x^2\sum_{q\geq\e^{\varrho_x}}\frac1{\{q-1\}^2}\bigg)\\
			&=z_\vartheta\big\{\log(1+t)+O\big(\e^{-\sqrt{\varrho_x}}+\varrho_x\e^{-\varrho_x}\big)\big\},
		\end{aligned}
	\end{equation}
	où la dernière égalité découle d'une application d'une forme forte du théorème de Mertens. Lorsque $1/(\log\varrho_x)^2-1\leq t\leq0$, nous obtenons
	\[\varphi(\vartheta,\varrho_x\{1+t\})-\varphi(\vartheta,\varrho_x)=z_\vartheta\big\{\log(1+t)+O\big(\e^{-\sqrt{\varrho_x}/\log\varrho_x}\big)\big\}.\]
	Définissons
	\begin{equation}\label{def:atheta;Atheta}
		a_\vartheta(t):=\frac{F_\omega(\e^{\varrho_x\{1+t\}},z_\vartheta)\e^{-\varrho_x t}}{1+t},\quad A_\vartheta(t):=\log a_{\vartheta}(t)\quad(\vartheta\in\Theta,\, t>-1).
	\end{equation}
	Nous déduisons de \eqref{eq:est:diff:Avartheta} que, pour $1/(\log\varrho_x)^2-1\leq t\leq (\log\varrho_x)^2-1$,
	\[A_\vartheta(t)-A_\vartheta(0)=(z_\vartheta-1)\log(1+t)-\varrho_x t+O\big(\varrho_x\e^{-\sqrt{\varrho_x}/\log\varrho_x}\big),\]
	soit
	\begin{equation}\label{eq:est:rapport:avartheta}
		\frac{a_\vartheta(t)}{a_\vartheta(0)}=(1+t)^{z_\vartheta-1}\e^{-\varrho_x t}\big\{1+O\big(\varrho_x\e^{-\sqrt{\varrho_x}/\log\varrho_x}\big)\big\}\quad \bigg(\frac{1}{\{\log\varrho_x\}^2}-1\leq t\leq \{\log\varrho_x\}^2-1\bigg).
	\end{equation}
	Posons
	\begin{equation*}\label{def:Ktheta}
		K_x(\vartheta):=\int_{1/(\log\varrho_x)^2-1}^{(\log\varrho_x)^2-1}(1+t)^{z_\vartheta-1}\e^{-\varrho_x t}\d t\quad(x\geq 3,\, \vartheta\in\Theta).
	\end{equation*}
	D'après \eqref{eq:Ic+:inter:2} et \eqref{eq:est:rapport:avartheta}, nous avons
	\begin{equation*}
		I_{\sC_x^+}(x)=\frac{\xi\e^{-\varrho_x}\{1+O(\varrho_x\e^{-\sqrt{\varrho_x}/\log\varrho_x})\}}{2\pi\varrho_x}\int_{\Theta}F_\omega(\e^{\varrho_x},z_\vartheta)\e^{\xi/z_\vartheta+i\vartheta}K_x(\vartheta)\d\vartheta.
	\end{equation*}
	Le changement de variables $v=\varrho_x(1+t)$ fournit
	\begin{equation}\label{eq:est:Ktheta}
		K_x(\vartheta)=\frac{\e^{\varrho_x}}{\varrho_x^{z_\vartheta}}\int_{\varrho_x/(\log\varrho_x)^2}^{\varrho_x(\log\varrho_x)^2}v^{z_\vartheta-1}\e^{-v}\d v=\frac{\{1+O(1/\varrho_x)\}\sqrt{2\pi}z_\vartheta^{z_\vartheta-1/2}\e^{\varrho_x-z_\vartheta}}{\varrho_x^{z_\vartheta}}\CommaBin
	\end{equation}
	où nous avons appliqué \eqref{eq:gen_incomplete_gamma} avec $\psi=\log^2$. Il vient,
	\begin{equation}\label{eq:eval:IC:inter:1}
		I_{\sC_x^+}(x)=\frac{\xi\{1+O(1/\varrho_x)\}}{\sqrt{2\pi}\varrho_x^{3/2}}\int_{\Theta}F_\omega(\e^{\varrho_x},z_\vartheta)\e^{\xi/z_\vartheta-z_\vartheta\{1+i\vartheta\}+3i\vartheta/2}\d\vartheta.
	\end{equation}
	\par Il reste à évaluer l'intégrale en $\vartheta$. Puisque $|{\e}^{-i\vartheta}-1|\leq1\ (\vartheta\in\Theta)$, nous avons 
	\begin{equation}\label{eq:eval:diff:varphi}
			\varphi(\vartheta,\varrho_x)-\varphi(0,\varrho_x)
			=(\e^{-i\vartheta}-1)\sum_{q<\e^{\varrho_x}}\frac{\varrho_x}{q-1+\varrho_x}+O\bigg(\varrho_x^2\vartheta^2\sum_{q<\e^{\varrho_x}}\frac{1}{q^2+\varrho_x^2}\bigg)\quad(\vartheta\in\Theta).
	\end{equation}
	Rappelons que, d'après \eqref{eq:estimation_rho}, $\varrho_x$ vérifie la relation
	\begin{equation*}\label{relation_col}
		\sum_{q<\e^{\varrho_x}}\frac{\varrho_x}{q-1+\varrho_x}=\frac{\xi}{\varrho_x}+O\big(\varrho_x{\e^{-\varrho_x}}\big).
	\end{equation*}
	En scindant la somme du terme d'erreur de \eqref{eq:eval:diff:varphi} à $q=\varrho_x$, nous obtenons
	\begin{equation*}\label{eq:eval:diff:varphi:2}
			\varphi(\vartheta,\varrho_x)-\varphi(0,\varrho_x)
			=\frac{\xi}{\varrho_x}\bigg\{{\e}^{-i\vartheta}-1+O\bigg(\frac{\vartheta^2}{\{\log\varrho_x\}^2}+\frac{|\vartheta|\e^{-\varrho_x}}{\log\varrho_x}\bigg)\bigg\},
	\end{equation*}
	puisque $\varrho_x^2\log\varrho_x\asymp\xi$, au vu de \eqref{eq:approximation:rho}.
	Définissons
	\[\Theta_x^+:=\bigg[-\Big(\frac{\xi}{\varrho_x}\Big)^{-1/3},\Big(\frac{\xi}{\varrho_x}\Big)^{-1/3}\bigg],\quad\Theta_x^-:=\Theta\smallsetminus\Theta_x^+\quad(x\geq 3).\]
	La formule intégrale de Cauchy implique
	\[\gs_{x,m}=\gs_m:=\frac1{(-i)^m}\bigg[\frac{\d^m\varphi(\vartheta,\varrho_x)}{\d\vartheta^m}\bigg]_{\vartheta=0}=\frac{\xi}{\varrho_x}\bigg\{1+O_{m}\bigg(\frac{1}{\{\log\varrho_x\}^2}\bigg)\bigg\}\quad(x\geq 3,\, m\geq 1).\]
	Puisque $\varphi'(0,\varrho_x)=-i\xi/\varrho_x+O(\varrho_x\e^{-\varrho_x})$, par définition de $\varrho_x$, il vient
	\[\varphi(\vartheta,\varrho_x)-\varphi(0,\varrho_x)=-\frac{i\xi\vartheta}{\varrho_x}-\tfrac12\gs_2\vartheta^2+\tfrac16i\gs_3\vartheta^3+O\bigg(\varrho_x|\vartheta|\e^{-\varrho_x}+\frac{\xi\vartheta^4}{\varrho_x}\bigg)\quad(\vartheta\in\Theta_x^+).\]
	Posons
	\[b(\vartheta):=F_\omega(\e^{\varrho_x},z_\vartheta)\e^{\xi/z_\vartheta-z_\vartheta\{1+i\vartheta\}+3i\vartheta/2},\quad B(\vartheta):=\log b(\vartheta)\quad(\vartheta\in\Theta).\]
	Il vient,
	\begin{align*}
		B(\vartheta)-B(0)&=\frac{\xi}{z_\vartheta}-z_\vartheta\{1+i\vartheta\}+\tfrac32i\vartheta-\frac{i\xi\vartheta}{\varrho_x}-\tfrac12\gs_2\vartheta^2+\tfrac16i\gs_3\vartheta^3-\frac{\xi}{\varrho_x}+\varrho_x+O\Big(\varrho_x|\vartheta|\e^{-\varrho_x}+\frac{\xi\vartheta^4}{\varrho_x}\Big)\\
		&=\tfrac32i\vartheta-\tfrac12\Big\{\frac{\xi}{\varrho_x}+\gs_2+\varrho_x\Big\}\vartheta^2-\tfrac16i\Big\{\frac\xi{\varrho_x}-\gs_3-2\varrho_x\Big\}\vartheta^3+O\Big(\varrho_x|\vartheta|\e^{-\varrho_x}+\frac{\xi\vartheta^4}{\varrho_x}\Big)\ (\vartheta\in\Theta_x^+),
	\end{align*}
d'où, en posant $\gs_2^*:=\xi/\varrho_x+\gs_2+\varrho_x$ et $\gs_3^*:=\xi/\varrho_x-\gs_3-2\varrho_x\ (x\geq 3)$,
	\[\frac{b(\vartheta)}{b(0)}=\e^{-\gs_2^*\vartheta^2/2}\big\{1+\tfrac32i\vartheta+O(\vartheta^2)\big\}\Big\{1+\tfrac16i\gs_3^*\vartheta^3+O\Big(\frac{\xi^2\vartheta^6}{\varrho_x^2}\Big)\Big\}\Big\{1+O\Big(\varrho_x|\vartheta|\e^{-\varrho_x}+\frac{\xi\vartheta^4}{\varrho_x}\Big)\Big\}\ (\vartheta\in\Theta_x^+).\]
	En intégrant sur $\Theta_x^+$, nous obtenons, grâce à la symétrie de l'intervalle d'intégration,
	\begin{equation}\label{eq:eval:Theta+}
		\begin{aligned}
			\int_{\Theta_x^+}\frac{b(\vartheta)\d\vartheta}{b(0)}&=\int_{\Theta_x^+}\e^{-\gs_2^*\vartheta^2/2}\Big\{1+O\Big(\vartheta^2+\frac{\xi\vartheta^4}{\varrho_x}+\frac{\xi^2\vartheta^6}{\varrho_x^2}\Big)\Big\}\d\vartheta\\
			&=\sqrt{\frac{2\pi}{\gs_2^*}}\Big\{1+O\Big(\frac1{\varrho_x}\Big)\Big\}=\sqrt{\frac{\pi\varrho_x}{\xi}}\Big\{1+O\Big(\frac1{\log\xi}\Big)\Big\}\cdot
		\end{aligned}
	\end{equation}
	Ainsi,
	\begin{equation}\label{eq:eval:Theta+:final}
		\int_{\Theta_x^+}F_\omega(\e^{\varrho_x},z_\vartheta)\e^{\xi/z_\vartheta-z_\vartheta\{1+i\vartheta\}+3i\vartheta/2}\d\vartheta=\frac{\sqrt{\pi\varrho_x}F_\omega(\e^{\varrho_x},\varrho_x)\e^{\xi/\varrho_x-\varrho_x}}{\sqrt\xi}\Big\{1+O\Big(\frac1{\log\xi}\Big)\Big\}\cdot
	\end{equation}
	Il reste à majorer
	\[\sE(x):=\int_{\Theta_x^-}F_\omega(\e^{\varrho_x},z_\vartheta)\e^{\xi/z_\vartheta-z_\vartheta\{1+i\vartheta\}+3i\vartheta/2}\d\vartheta\quad(x\geq 3).\]
	En majorant $|F_\omega(\e^{\varrho_x},\varrho_x\e^{-i\vartheta})|$ comme dans \cite[lemme 3]{bib:erdos}, nous déduisons l'existence d'une constante absolue $c>0$ telle que
	\[\frac{|F_\omega(\e^{\varrho_x},\varrho_x\e^{-i\vartheta})|}{F_\omega(\e^{\varrho_x},\varrho_x)}\ll\e^{-c\xi\vartheta^2/\varrho_x}\quad(\vartheta\in\Theta).\]
	Ainsi,
	\begin{align*}
		|{\sE}(x)|&\ll F_\omega(\e^{\varrho_x},\varrho_x)\int_{\Theta_x^-}\e^{-c\xi\vartheta^2/\varrho_x}\e^{\Re\{\xi/z_\vartheta-z_\vartheta(1+i\vartheta)\}}\d\vartheta\\
		&\ll F_\omega(\e^{\varrho_x},\varrho_x)\e^{\xi/\varrho_x-\varrho_x}\int_{\Theta_x^-}\e^{-c\xi\vartheta^2/\varrho_x}\e^{\xi(\cos\vartheta-1)/\varrho_x+\varrho_x(1-\vartheta\sin\vartheta-\cos\vartheta)}\d\vartheta.
	\end{align*}
	Puisque $1-\vartheta\sin\vartheta-\cos\vartheta\sim-\frac12\vartheta^2<0$ $(\vartheta\in\Theta_x^-$), il vient
	\begin{equation}\label{eq:eval:Theta-}
		|{\sE}(x)|\ll F_\omega(\e^{\varrho_x},\varrho_x)\e^{\xi/\varrho_x-\varrho_x}\int_{\Theta_x^-}\e^{-c\xi\vartheta^2/\varrho_x}\d\vartheta\ll F_\omega(\e^{\varrho_x},\varrho_x)\e^{\xi/\varrho_x-\varrho_x}\Big(\frac{\varrho_x}{\xi}\Big)^{3/2}\cdot
	\end{equation}
	En regroupant les estimations \eqref{eq:eval:IC:inter:1}, \eqref{eq:eval:Theta+:final} et \eqref{eq:eval:Theta-}, nous obtenons bien \eqref{eq:formule:IC+}.
\end{proof}




\section{Preuve du Théorème \ref{thm:somme_inverse:sans_multiplicite}: estimation des termes d'erreur}\label{sec:errorterm}


\begin{lemma}
	Nous avons
	\begin{equation}\label{eq:majoration:IC-}
		I_{\sC_x^-}(x)\ll F_\omega(\e^{\varrho_x},\varrho_x)\e^{\xi/2\varrho_x-\varrho_x}\xi^{1/4}(\log\xi)^{3/4}\quad(x\geq 10^7).
	\end{equation}
\end{lemma}
\begin{proof}
	Rappelons la notation $z_\vartheta=\varrho_x\e^{-i\vartheta}$ et posons 
	\[T:=\{\vartheta:\pi/3<|\vartheta|\leqslant \pi/2\},\quad J_1(u):=\int_{T}F_\omega(\e^u,z_\vartheta)\e^{\xi/z_\vartheta+i\vartheta}\d\vartheta\quad(u\in\sU_x).\]
	D'après la définition de $I_{\sC_x^-}$ en \eqref{def:gJ;IE}, nous avons
	\begin{equation}\label{eq:redef:IC-}
		I_{\sC_x^-}(x)=\frac{\xi}{2\pi\varrho_x}\int_{\sU_x}\frac{J_1(u)\e^{-u}\d u}{u}\cdot
	\end{equation}
	Notons que, pour $\vartheta\in T$,
	\[\big|F_\omega(\e^u,z_\vartheta)\big|=\prod_{q<\e^u}\bigg|1+\frac{\varrho_x\e^{-i\vartheta}}{q-1}\bigg|\leq\prod_{q<\e^u}\Big(1+\frac{\varrho_x}{q-1}\Big)=F_\omega(\e^u,\varrho_x).\]
	Ainsi,
	\begin{equation}\label{eq:majo:E1}
		|J_1(u)|\leq F_\omega(\e^{u},\varrho_x)\int_T\e^{\xi\cos\vartheta/\varrho_x}\d\vartheta\ll F_\omega(\e^u,\varrho_x)\e^{\xi/2\varrho_x}.
	\end{equation}
	Posant
	\begin{equation}\label{eq:def:E2}
		J_2(x):=\int_{\sU_x}\frac{F_\omega(\e^{u},\varrho_x)\e^{-u}\d u}{u}\quad(x\geq 3),
	\end{equation}
	nous déduisons de \eqref{eq:redef:IC-} et \eqref{eq:majo:E1} la majoration
	\begin{equation}\label{eq:majo:moduleIc-:inter:1}
		I_{\sC_x^-}(x)\ll\e^{\xi/2\varrho_x}J_2(x)\sqrt{\xi\log\xi}.
	\end{equation}
	Rappelons la définition de $a_\vartheta$ en \eqref{def:atheta;Atheta}. À l'aide du changement de variables $u=\varrho_x(1+t)$, nous obtenons
	\[J_2(x)=\e^{-\varrho_x}\int_{1/(\log\varrho_x)^2-1}^{(\log\varrho_x)^2-1}\frac{F_\omega(\e^{\varrho_x\{1+t\}},\varrho_x)\e^{-\varrho_x t}\d t}{1+t}=\e^{-\varrho_x}\int_{1/(\log\varrho_x)^2-1}^{(\log\varrho_x)^2-1}a_0(t)\d t.\]
	L'estimation \eqref{eq:est:rapport:avartheta} appliquée avec $\vartheta=0$, fournit
	\begin{equation}\label{eq:majo:error_term:E2}
		J_2(x)\ll F_\omega(\e^{\varrho_x},\varrho_x)\e^{-\varrho_x}K_x(0)\ll \frac{F_\omega(\e^{\varrho_x},\varrho_x)\e^{-\varrho_x}}{\sqrt{\varrho_x}}\CommaBin
	\end{equation}
	d'après \eqref{eq:est:Ktheta}. La formule \eqref{eq:majoration:IC-} est alors une conséquence directe des majorations \eqref{eq:majo:moduleIc-:inter:1} et \eqref{eq:majo:error_term:E2}.
\end{proof}


Il reste à évaluer $I_{\sD_x^{\pm}}$. Rappelons que
\[I_{\sD_x^{\pm}}(x)=\frac1{2\pi i}\int_{\sD_x^{\pm}}\e^s\d s\int_{\sU_x}\frac{F_\omega(\e^u,\xi/s)\e^{-u}\d u}{u}\quad(x\geq 3).\]


\begin{lemma}
	Nous avons
	\begin{equation}\label{eq:majoration:ID+-}
		|I_{\sD_x^{\pm}}(x)|\ll \frac{F_\omega(\e^{\varrho_x},\varrho_x)\e^{-\varrho_x}}{\sqrt{\varrho_x}}\quad(x\geq 3).
	\end{equation}
\end{lemma}
\begin{proof}
	Notons d'emblée que
	\[\sD_x^{\pm}=\Big\{t\pm \frac{i\xi}{\varrho_x}:t\leq 0\Big\}\]
	et rappelons la définition de $J_2(x)$ en \eqref{eq:def:E2}. Nous affirmons que
	\begin{equation}\label{eq:majo:I:D^+:1}
		I_{\sD^{\pm}}(x)\ll J_2(x).
	\end{equation}
	En effet, puisque l'on a, uniformément pour $t\leq 0$,
	\begin{align*}
		\bigg|F_\omega\Big({\e^u},\frac{\xi}{t\pm i\xi/\varrho_x}\Big)\bigg|&=\prod_{q<\e^u}\bigg|1+\frac{\xi}{(t\pm i\xi/\varrho_x)(q-1)}\bigg|\leq\prod_{q<\e^u}\Big(1+\frac{\varrho_x}{q-1}\Big)=F_\omega(\e^u,\varrho_x),
	\end{align*}
	 nous pouvons écrire
	\[\bigg|\int_{\sD_x^{\pm}}F_\omega\Big({\e^u},\frac{\xi}{s}\Big)\e^s\d s\bigg|\leq\int_{-\infty}^0\bigg|F_\omega\Big({\e^u},\frac{\xi}{t\pm i\xi/\varrho_x}\Big)\bigg|\e^t\d t\leq F_\omega(\e^u,\varrho_x)\quad (u\in\sU_x),\]
	et donc \eqref{eq:majo:I:D^+:1}. Le résultat découle alors des estimations \eqref{eq:majo:error_term:E2} et \eqref{eq:majo:I:D^+:1}.
\end{proof}




\section{Preuve du Théorème \ref{thm:somme_inverse:sans_multiplicite}: complétion de l'argument}\label{sec:proof;th1}

D'après \eqref{eq:majoration:IC-} et \eqref{eq:majoration:ID+-}, il est clair que
\begin{equation*}\label{eq:compare:D;C-}
	\max\big\{I_{\sD_x^{\pm}}(x), I_{\sC_x^-}(x)\big\}\ll F_\omega(\e^{\varrho_x},\varrho_x)\e^{\xi/2\varrho_x-\varrho_x}\xi^{1/4}(\log\xi)^{3/4}.
\end{equation*}
Les estimations \eqref{eq:formule:IC+} et \eqref{eq:majoration:IC-} fournissent de plus
\begin{equation*}\label{eq:compare:C-;C+}
	\max\big\{I_{\sD_x^{\pm}}(x), I_{\sC_x^-}(x)\big\}\ll \frac{I_{\sC_x^+}(x)\sqrt{\xi}\e^{-\xi/2\varrho_x}}{\sqrt{\varrho_x}}\ll \frac{I_{\sC_x^+}(x)}{\log\xi},
\end{equation*}
puisque, d'après \eqref{eq:approximation:rho}, nous avons $\varrho_x\leq 2\sqrt{\xi/\log\xi}$ pour $x$ suffisamment grand, soit $-\xi/2\varrho_x\leq -\tfrac14\sqrt{\xi\log\xi}$.
Les estimations \eqref{eq:formule:Siotasstar} et \eqref{eq:formule:IC+} impliquent donc
\begin{equation}\label{eq:eval:S_iota**:final}
	S_{\omega,\iota}^{**}(x)=\frac{xI_{\sC_x^+}(x)}{\xi\log x}\Big\{1+O\Big(\frac1{\log \xi}\Big)\Big\}=\frac{xF_\omega(\e^{\varrho_x},\varrho_x)\e^{\xi/\varrho_x-\varrho_x}}{\sqrt 2\varrho_x(\log x)\sqrt{\log_2 x}}\Big\{1+O\Big(\frac1{\log\xi}\Big)\Big\}\cdot
\end{equation}
\par Il reste à estimer la somme $S_{\omega,\pi}(x)$ correspondant au cas des entiers $n\leq x$  pour lesquels $\omega(n)$ est pair.
Posons
\[S_{\omega,\pi}^*(x):=\sum_{p\in\sP_{\omega,x}}\frac1p\sum_{k\in\sK_{\omega,x}}\sum_{\substack{a\leq x/p\\P^+(a)<p\\\omega(a)=k}}\sum_{1\leq\ell\leq10\log_2 x}\Phi_{\omega,k+1}\Big(\frac x{ap^\ell},p\Big)\quad(x\geq 16).\]
L'estimation \eqref{eq:est:S;Sstar} persiste alors sous la forme
\[S_{\omega,\pi}(x)=S_{\omega,\pi}^*(x)+O\Big(\frac{x}{\log x}\exp\Big\{{\sqrt{2(\log_2 x)\log_3 x}\Big(1-\frac{c_1\log_4 x}{\log_3 x}\Big)}\Big\}\Big)\quad\big(x\geq  10^7\big),\]
pour une constante $c_1>\tfrac32$. L'ensemble des estimations obtenues dans le cas $\omega(n)$ impair est encore valide lorsque $\omega(n)$ est pair. La seule différence significative réside dans la formule \eqref{eq:eval:Siotasstar:hankel:sum} pour laquelle la quantité $\lambda_\omega(p,k)$ doit être remplacée par $\lambda_\omega(p,k-1)$. Définissant
\[S_{\omega,\pi}^{**}(x):=\frac{x}{2\pi i\log x}\int_{\sU_x}\frac{\e^{-u}\d u}{u}\int_{\sH_x}F_\omega\Big({\e^u},\frac{\xi}s\Big)\frac{\e^s}{s}\d s\quad(x\geq 3),\]
nous avons l'estimation
\[S_{\omega,\pi}^*(x)=S_{\omega,\pi}^{**}(x)\Big\{1+O\Big(\frac1{\log_3 x}\Big)\Big\}\quad(x\geq 3).\]
Ainsi, rappelant que $s\in\sC_x^+\Rightarrow |s|=\xi/\varrho_x$, nous obtenons finalement
\begin{equation}\label{eq:compare:pi:iota}
	S_{\omega,\pi}(x)=\varrho_x S_{\omega,\iota}(x)\Big\{1+O\Big(\frac1{\log_3 x}\Big)\Big\}\quad(x\geq 3).
\end{equation}
Par conséquent, la formule \eqref{eq:somme_inverse:sans_multiplicite} découle des estimations \eqref{eq:est:S;Sstar}, \eqref{eq:reduction:Sstar}, \eqref{eq:eval:S_iota**:final} et \eqref{eq:compare:pi:iota}.\qed


\section{Formule explicite pour $\log S_\omega(x)$ et comportement local}\label{sec:explicit_formula}


\subsection{Développement asymptotique de $\log F_\omega(\e^{\varrho_x},\varrho_x)$}

Rappelons que, d'après la définition de $F_\omega(y,z)$ en \eqref{def:Hnu;H0}, nous avons
\[\log F_\omega(\e^{\varrho_x},\varrho_x)=\sum_{q<\e^{\varrho_x}}\log\Big(1+\frac{\varrho_x}{q-1}\Big).\]
Afin d'évaluer cette somme, nous la scindons à $q=\varrho_x$. Posons
\[\beta_m:=(-1)^m(m-1)!\sum_{n\geq 1}\frac{(-1)^{n+1}}{n^m(n+1)}\quad(m\geq 1).\]


\begin{lemma}
	Soit $M\geq 1$. Nous avons
	\begin{equation}\label{eq:AsEx:logF:big}
		\sum_{v<q\leq\e^v}\log\Big(1+\frac{v}{q}\Big)=v\bigg\{\log\Big(\frac{v}{\log v}\Big)+\sum_{1\leq m\leq M}\frac{\beta_m}{(\log v)^m}+O_{M}\bigg(\frac{1}{\{\log v\}^{M+1}}\bigg)\bigg\}\cdot
	\end{equation}
\end{lemma}
\begin{proof}
	Posons
	\[I_1( v):=\int_{ v}^{\e^{ v}}{\frac{\log(1+ v/t)\d t}{\log t}}\quad(v\geq 3).\]
	Une application du théorème des nombres premiers fournit
	\begin{equation}\label{eq:cut:sumlow:pnt}
		\sum_{ v<q\leq\e^ v}\log\Big(1+\frac{ v}{q}\Big)=I_1( v)+O\bigg(\int_{ v}^{\e^{ v}}\log\Big(1+\frac{ v}{t}\Big)\e^{-\sqrt{\log t}}\d t\bigg).
	\end{equation}
	Le terme d'erreur peut être majoré par
	\begin{equation}\label{eq:majo:error:logF:big:pnt}
		v\int_{ v}^{\infty}\frac{\e^{-\sqrt{\log t}}}t\d t\ll v\sqrt{\log v}\e^{-\sqrt{\log v}}.
	\end{equation}
	Par ailleurs, d'après \eqref{eq:an_integral}, nous avons
	\begin{equation}\label{eq:est:I3:inter}
		\begin{aligned}
			I_1( v)&=\int_{ v}^{\e^{ v}}\frac{\d t}{\log t}\sum_{n\geq 1}\frac{(-1)^{n+1} v^n}{nt^n}=\sum_{n\geq 1}\frac{(-1)^{n+1} v^n}{n}\int_ v^{\e^{ v}}\frac{\d t}{t^n\log t}\\
			&= v\bigg\{\log\Big(\frac{ v}{\log v}\Big)+\sum_{1\leq m\leq M}\frac{(-1)^m(m-1)!}{(\log v)^m}\sum_{n\geq 1}\frac{(-1)^{n+1}}{(n+1)n^m}+O_{M}\bigg(\frac{1}{\{\log v\}^{M+1}}\bigg)\bigg\}\cdot
		\end{aligned}
	\end{equation}
	En regroupant les estimations \eqref{eq:cut:sumlow:pnt}, \eqref{eq:majo:error:logF:big:pnt} et \eqref{eq:est:I3:inter}, nous obtenons le résultat annoncé.
\end{proof}


Posons, pour $m\geq 1$,
\begin{gather*}
	\gamma_m:=m!\bigg\{1+\frac1m\sum_{n\geq 1}\frac{(-1)^{n+1}}{n(n+1)^m}\bigg\},\quad b_m:=\gamma_m+\beta_m=(m-1)!\sum_{0\leq j\leq m/2}\frac{\alpha_{2j}}{(2j-1)!}\cdot
\end{gather*}

\begin{lemma}
	Soit $M\geq 1$. Nous avons
	\begin{equation}\label{eq:AsEx:logF:low}
		\sum_{q\leq v}\log\Big(1+\frac{ v}{q}\Big)= v\bigg\{\sum_{1\leq m\leq M}\frac{\gamma_m}{(\log v)^m}+O_{M}\bigg(\frac{1}{\{\log v\}^{M+1}}\bigg)\bigg\}\cdot
	\end{equation}
\end{lemma}
\begin{proof}
	Posons
	\[I_2( v):=\int_2^ v\frac{\log(1+ v/t)\d t}{\log t}\quad (v\geq 2),\]
	de sorte que
	\begin{equation}\label{eq:cut:sumbig:pnt}
		\sum_{q\leq v}\log\Big(1+\frac{ v}{q}\Big)=I_2( v)+\int_{2}^{ v}\log\Big(1+\frac{ v}{t}\Big)\d\{\pi(t)-\li(t)\}.
	\end{equation}
	Or, l'intégrale du membre de droite de \eqref{eq:cut:sumbig:pnt} peut être majorée par
	\begin{equation}\label{eq:majo:error:logF:low:pnt}
		\bigg[t\e^{-\sqrt{\log t}}\log\Big(1+\frac{ v}{t}\Big)\bigg]_2^{ v}+ v\int_2^ v\frac{\e^{-\sqrt{\log t}}\d t}{t+ v}\ll  v\e^{-\sqrt{\log v}}+\int_2^ v\e^{-\sqrt{\log t}}\d t\ll  v\e^{-\sqrt{\log v}}.
	\end{equation}
	De plus, d'après \eqref{eq:a_second_integral}, nous avons
	\begin{equation}\label{eq:est:I4:inter}
		\begin{aligned}
			I_2( v)&=(\log v)\li(v)- v+\int_2^ v\frac{\d t}{\log t}\sum_{n\geq 1}\frac{(-1)^{n+1}t^n}{n v^n}+O(1)\\
			&= v\bigg\{\sum_{1\leq m\leq M}\frac{m!}{(\log v)^m}+\sum_{n\geq 1}\frac{(-1)^{n+1}}{n v^{n+1}}\int_2^{ v}\frac{t^n\d t}{\log t}+O_{M}\bigg(\frac{1}{\{\log v\}^{M+1}}\bigg)\bigg\}\\
			&= v\sum_{1\leq m\leq M}\Bigg\{\frac{m!}{(\log v)^m}+\frac{(m-1)!}{(\log v)^m}\sum_{n\geq 1}\frac{(-1)^{n+1}}{n(n+1)^m}+O_{M}\bigg(\frac{1}{\{\log v\}^{M+1}}\bigg)\Bigg\}\cdot
		\end{aligned}
	\end{equation}
	En regroupant les estimations \eqref{eq:cut:sumbig:pnt}, \eqref{eq:majo:error:logF:low:pnt} et \eqref{eq:est:I4:inter}, nous obtenons \eqref{eq:AsEx:logF:low}. Cela termine la démonstration.
\end{proof}



\begin{proposition}
	Soit $M\geq 1$. Nous avons
	\begin{equation}\label{eq:AsEx:logF}
		\log F_\omega(\e^{\varrho_x},\varrho_x)=\varrho_x\bigg\{\log\Big(\frac{\varrho_x}{\log\varrho_x}\Big)+\sum_{1\leq m\leq M}\frac{b_m}{(\log\varrho_x)^m}+O_{M}\bigg(\frac{1}{\{\log\varrho_x\}^{M+1}}\bigg)\bigg\}\cdot
	\end{equation}
\end{proposition}
\begin{proof}
	Nous avons directement
	\begin{equation}\label{eq:computing:phi_rho}
			\log F_\omega(\e^{\varrho_x},\varrho_x)=\sum_{q\leq\varrho_x}\log\Big(1+\frac{\varrho_x}{q}\Big)+\sum_{\varrho_x<q< \e^{\varrho_x}}\log\Big(1+\frac{\varrho_x}{q}\Big)+O(\log\varrho_x).
	\end{equation}
	puis \eqref{eq:AsEx:logF} en évaluant la première somme de \eqref{eq:computing:phi_rho} par \eqref{eq:AsEx:logF:low} et la deuxième par \eqref{eq:AsEx:logF:big}. 
\end{proof}



\subsection{Preuve du Corollaire \ref{cor:explicit_formula}}
Posons
\begin{gather*}
	\gE_N(x):=\frac{\sqrt{2\log_2 x}(\log_4 x)^{N+1}}{(\log_3 x)^{N+1/2}}\quad\big(x\geq  10^7\big),\quad\sG_\omega(v):=\log F_\omega(\e^v,v)+\frac{\xi}{v}-v\quad(v\geq 2),\\
	c_1=-\tfrac12,\quad c_n:=\tfrac12\big(b_{n-1}+\alpha_{n-1}\textbf{1}_{\{n\equiv 1\;\text{mod}\, 2\}}\big)\quad(n\geq 2).
\end{gather*}
Remarquons d'emblée que $\log_3 x=\e^{O(\gE_N(x))}$, de sorte que, d'après \eqref{eq:somme_inverse:sans_multiplicite}, nous avons
\begin{equation*}\label{eq:S(x):explicit:inter:1}
	\log S_\omega(x)=\log\frac{x}{\log x}+\sG_\omega(\varrho_x)+O\big({\gE_N(x)}\big).
\end{equation*}
Or, d'après \eqref{eq:AsEx:logF}, nous avons
\[\sG_\omega(\varrho_x)\approx2\varrho_x\log\varrho_x\bigg\{1-\frac{\log_2\varrho_x}{\log\varrho_x}+\sum_{n\geq 1}\frac{c_n}{(\log\varrho_x)^{n}}\bigg\}\cdot\]
Rappelons les définitions \eqref{def:series_formelles} et introduisons les séries formelles
\[L(s,t):=1-t+2sH(s,t),\quad Q(s,t):=E(s,t)L(s,t)\bigg\{1-\frac{2\{t+s\log L(s,t)\}}{L(s,t)}+\sum_{n\geq 1}c_n\Big(\frac{2s}{L(s,t)}\Big)^n\bigg\}.\]
Nous déduisons directement de développements en série successifs, l'existence de réels $\{z_{j,k}\}_{j,k\geq 0}$ tels que
\[Q(s,t)=\sum_{j,k\geq 0}z_{j,k}s^jt^k.\]
Remarquant que
\[\varrho_x\approx\sqrt{\frac{2\xi}{\log\xi}}E(\sigma_x,\tau_x),\quad \log \varrho_x\approx\frac{L(\sigma_x,\tau_x)}{2\sigma_x},\quad\log_2\varrho_x\approx\frac{\tau_x}{\sigma_x}+\log L(\sigma_x,\tau_x),\]
nous obtenons
\[\sG_\omega(\varrho_x)\approx\sqrt{2\xi\log\xi}Q(\sigma_x,\tau_x).\]
En regroupant les termes de même degré, nous en déduisons l'existence de polynômes $P_j\ (j\geq 0)$ tels que
\[\sG_\omega(\varrho_x)\approx\sqrt{2\xi\log\xi}\sum_{j\geq 0}\frac{P_j(\log_2\xi)}{(\log\xi)^j},\]
vérifiant
\begin{equation}\label{def:polynomesPj:expl}
	P_j(X):=\sum_{0\leq k\leq j}z_{j-k,k}(X-\log 2)^k
\end{equation}
et donc \eqref{eq:formule_explicite:S}. Remarquons enfin que
\[[X^j]P_j(X)=z_{0,j}=[t^j]Q(0,t)=[t^j]\sqrt{1-3t}=\frac{3^j(2j)!}{4^j(1-2j)(j!)^2}\cdot\]
Cela achève la démonstration.



\subsection{Preuve du Corollaire \ref{cor:local_behaviour}}
Rappelons la définition de $\varrho_x$ en \eqref{def:Hnu;H0} et posons
\[\delta_x:=\frac1{\log_3 x},\quad \eta_x:=\frac{\sqrt{\log_2 x}}{(\log_3 x)^{3/2}}\quad (x\geq 16).\]
Commençons par préciser le comportement local du paramètre $\varrho_x$. Afin d'alléger l'écriture, nous notons dans la suite $X:=x^{1+h}\ (h\in\R)$.


\begin{lemma}
	Soient $0<\varepsilon<1$ et $h>-1+\varepsilon$. Sous les conditions $x\geq 10^{7/\varepsilon}$, $\log(1+h)\ll\eta_x$, nous avons
	\begin{equation}\label{eq:delta_varrho}
		\varrho_{X}-\varrho_x\ll\delta_x^2.
	\end{equation}
\end{lemma}
\begin{proof}
	Rappelons la définition de $\Psi_x(v)$ en \eqref{def:Hnu;H0} et remarquons d'emblée que
	\[\Psi_{X}(v)-\Psi_x(v)=\frac{\log(1+h)}{v^2}\quad(v>0),\]
	de sorte que, d'après \eqref{eq:estimation_rho}, nous avons
	\begin{equation}\label{eq:psi_hx-psi_x}
		\Psi_x(\varrho_X)=-\frac{\log(1+ h)}{\varrho_{X}^2}+O(\e^{-\varrho_X}).
	\end{equation}
	Par ailleurs, un rapide calcul permet d'obtenir
	\[\Psi_x(v+\gv)-\Psi_x(v)=-\gv\Big\{\frac{2\xi}{v^3}+\frac1v+O\Big(\frac1{v\log v}\Big)\Big\}\quad(\gv\ll 1),\]
	soit, avec $\gv=|\varrho_{X}-\varrho_x|$ et pour $x$ suffisamment grand,
	\begin{equation}\label{eq:delta_varrho_hx;x}
		|\varrho_{X}-\varrho_x|\leq\frac{\varrho_{x}^3|\Psi_x(\varrho_x)-\Psi_x(\varrho_X)|}{2\xi}\leq\frac{\varrho_x^3}{2\xi}\bigg\{\frac{|\log(1+h)|}{\varrho_{X}^2}+O\big({\e^{-\min\{\varrho_{X},\varrho_{x}\}}}\big)\bigg\}\CommaBin
	\end{equation}
	d'après \eqref{eq:estimation_rho} et \eqref{eq:psi_hx-psi_x}. En particulier, puisque d'après \eqref{eq:approximation:rho}, $\varrho_{X}\sim\varrho_x$ lorsque $\log(1+h)\ll\eta_x$, nous déduisons \eqref{eq:delta_varrho} de \eqref{eq:delta_varrho_hx;x}.
\end{proof}
Nous pouvons désormais compléter la preuve du Corollaire \ref{cor:local_behaviour}. Pour $\log(1+h)\ll \eta_x$, nous avons
\begin{equation*}\label{eq:estimates:utils}
	\sqrt{2\log_2X}=\{1+O(\delta_x)\}\sqrt{2\log_2 x},\quad\frac{\log_2X}{\varrho_{X}}=\frac{\log_2x}{\varrho_{X}}+O(\delta_x),\\
\end{equation*}
Nous déduisons alors de \eqref{eq:somme_inverse:sans_multiplicite} que
\begin{equation}\label{eq:eval_rapport;S(hx)/S(x)}
	\frac{(1+h)S_\omega(X)}{x^h S_\omega(x)}=\frac{\{1+O(\delta_x)\}F_\omega(\e^{\varrho_{X}},\varrho_{X})\e^{\xi/\varrho_{X}-\varrho_{X}}}{F_\omega(\e^{\varrho_{x}},\varrho_{x})\e^{\xi/\varrho_{x}-\varrho_{x}}}\cdot
\end{equation}
Par ailleurs, d'après \eqref{eq:delta_varrho}, nous avons d'une part
\begin{equation}\label{eq:delta_xi_over_rho}
	\frac{\xi}{\varrho_{X}}-\frac{\xi}{\varrho_x}\ll\frac{\xi|\varrho_{X}-\varrho_x|}{\varrho_x^2}\ll\delta_x,
\end{equation}
et, d'autre part, notant $\varrho^-:=\min(\varrho_x,\varrho_{X})$ et $\varrho^+:=\max(\varrho_x,\varrho_{X})$,
\begin{equation}\label{eq:delta_logFhx_x}
	\begin{aligned}
		\log\frac{F_\omega(\e^{\varrho_{X}},\varrho_{X})}{F_\omega(\e^{\varrho_{x}},\varrho_{x})}&=\sum_{q\leq\exp(\varrho^-)}\log\Big(1+\frac{\varrho_{X}-\varrho_x}{q-1+\varrho_x}\Big)+\sum_{\exp(\varrho^-)< q<\exp(\varrho^+)}\log\Big(1+\frac{\varrho^+}{q-1}\Big)\\
		&\ll |\varrho_{X}-\varrho_x|\log\varrho_x\ll\delta_x.
	\end{aligned}
\end{equation}
L'estimation \eqref{eq:asymptotics:Sx} résulte alors des estimations \eqref{eq:eval_rapport;S(hx)/S(x)}, \eqref{eq:delta_xi_over_rho}, \eqref{eq:delta_logFhx_x} et de la majoration \eqref{eq:delta_varrho}.\qed



\section{Preuve du Théorème \ref{thm:somme_inverse:avec_mult}}\label{sec:proof;th4}


\subsection{Préparation}

Rappelons la définition de $F_\Omega(y,z)$ en \eqref{def:Hnu;H0}, et celle de $g_\Omega(y,z)$ en \eqref{def:g}.


\begin{lemmanodot}{\bf (\cite[th. 1]{bib:alladi}). }\label{l:alladi}
	Sous les conditions $y\leq\e^{(\log x)^{2/5}}$, $|z|< 2$, nous avons
	\begin{equation}\label{eq:alladi}
		\sum_{\substack{m\leq x\\ P^-(m)\geq y}}z^{\Omega(m)}=\frac{xg_\Omega(y,z)}{\Gamma(z)(\log x)^{1-z}}+O\bigg(\frac{x(\log y)^{1-\Re z}}{(\log x)^{2-\Re z}}\bigg)\quad (x\geq 3).
	\end{equation}
\end{lemmanodot}


Posons
\begin{equation}\label{def:lambdaxpz}
	\lambda_{\Omega,x}(y,z):=\sum_{\substack{P^+(m)\leq y\\ \Omega(m)\leq3\xi/2}}\frac{z^{-\Omega(m)}}{m}\quad\big(2\leq y\leq x,\;|z|>\tfrac12\big).
\end{equation}


\begin{lemma}
	Sous les conditions $2\leq y\leq \log x$, $|z|=1$, nous avons
	\begin{equation}\label{eq:est_ser_gen_a}
		\lambda_{\Omega,x}(y,z)=F_\Omega\Big(y,\frac1z\Big)+O\Big(\frac{1}{\log x}\Big)\cdot
	\end{equation}
\end{lemma}
\begin{proof}
	Pour tout $k\geq 1$, posons $v_k:=2-1/k$. Une majoration de type Rankin fournit
	\[\sum_{\substack{P^+(m)\leq y\\ \Omega(m)>3\xi/2}}\frac{z^{-\Omega(m)}}{m}\ll\sum_{k>3\xi/2}\sum_{P^+(m)\leq y}\frac{v_k^{\Omega(m)-k}}{m}\ll\sum_{k>3\xi/2}\frac{k}{v_k^{k}}\prod_{\substack{3\leq q\leq y}}\Big(1-\frac{v_k}{q}\Big)^{-1}\ll\frac{\xi(\log y)^2}{2^{3\xi/2}}\CommaBin\]
	et donc \eqref{eq:est_ser_gen_a}, d'après \eqref{def:lambdaxpz}, en observant que $\frac32\log 2>1$.
\end{proof}



\subsection{Complétion}\label{subsec:completionOm}

Notons que, pour $p\leq\log x$ et $\Omega(A)\in\sK_{\Omega,x}$, nous avons
\[Ap\leq (\log x)^{3(\log_2 x)/2+1}\Rightarrow p\leq\log x\leq\e^{\{\log(x/Ap)\}^{2/5}}.\]
Nous déduisons donc de \eqref{eq:alladi} que, pour $p\leq \log x$, $\Omega(A)\in\sK_{\Omega,x}$ et $|z|=1$,
\begin{equation}\label{eq:est:sum_b:alladi}
	\sum_{\substack{B\leq x/Ap\\ P^-(B)\geq p}}z^{\Omega(B)}=\frac{xg_\Omega(p,z)}{Ap\Gamma(z)\{\log(x/Ap)\}^{1-z}}+O\bigg(\frac{x(\log p)^{1-\Re z}}{Ap\{\log (x/Ap)\}^{2-\Re z}}\bigg)\cdot
\end{equation}
En remarquant d'une part que $g_\Omega(p,z)\asymp(\log p)^{-\Re z}$ et d'autre part que 
\[\log\frac{x}{Ap}=\bigg\{1+O\bigg(\frac{(\log_2 x)^2}{\log x}\bigg)\bigg\}\log x,\]
l'estimation \eqref{eq:est:sum_b:alladi} implique
\begin{equation}\label{eq:est:sum_b:final}
	\sum_{\substack{B\leq x/Ap\\ P^-(B)\geq p}}z^{\Omega(B)}=\frac{xg_\Omega(p,z)}{Ap\Gamma(z)(\log x)^{1-z}}\bigg\{1+O\bigg(\frac{(\log_2 x)^2}{\log x}\bigg)\bigg\}\cdot 
\end{equation}
Rappelons la définition de $\lambda_{\Omega,x}(y,z)$ en \eqref{def:lambdaxpz}. Nous déduisons de la définition de $S_{\Omega,\iota}^*(x)$ en \eqref{def:Siotastar} et de l'estimation \eqref{eq:est:sum_b:final} que
\begin{equation}\label{eq:est:Siotastar:afterb}
	S_{\Omega,\iota}^*(x)=\bigg\{\frac1{2\pi i}+O\bigg(\frac{(\log_2 x)^2}{\log x}\bigg)\bigg\}\sum_{p\leq\log x}\frac x{p^2}\oint_{|z|=1}\frac{g_\Omega(p,z)\lambda_{\Omega,x}(p,z)\d z}{\Gamma(1+z)(\log x)^{1-z}}\quad (x\geq 3).
\end{equation}
Posons
\[\sG_\Omega(y,z):=\frac{g_\Omega(y,z)F_\Omega(y,1/z)}{\Gamma(1+z)}
\quad(y\geq 2,\, \tfrac12<|z|<2).\]
D'après \eqref{eq:est_ser_gen_a}, nous avons
\begin{equation}\label{eq:est:Siotastar:aftera}
	\begin{aligned}
		S_{\Omega,\iota}^*(x)&=\bigg\{\frac{1}{2\pi i}+O\bigg(\frac{(\log_2x)^2}{\log x}\bigg)\bigg\}\sum_{p\leq\log x}\frac{x}{p^2}\bigg\{\oint_{|z|=1}\frac{\sG_\Omega(p,z)\d z}{(\log x)^{1-z}}+O\bigg(\oint_{|z|=1}\frac{|g_\Omega(p,z){\rm d}z|}{(\log x)^{2-\Re z}}\bigg)\bigg\}\\
		&=\bigg\{\frac{1}{2\pi i}+O\bigg(\frac{(\log_2x)^2}{\log x}\bigg)\bigg\}\sum_{p\leq\log x}\frac{x}{p^2}\bigg\{\oint_{|z|=1}\frac{\sG_\Omega(p,z)\d z}{(\log x)^{1-z}}+O\Big(\frac{\log_2 x}{\log x}\Big)\bigg\},
	\end{aligned}
\end{equation}
puisque $|g_\Omega(p,z)|\asymp(\log p)^{-\Re z}$.

Fixons $J\geq 1$ et posons $\delta=\delta_J:=\tfrac12(1/{\gp_{J+1}}+1/{\gp_{J+2}})$. Le théorème des résidus fournit
\begin{equation}\label{eq:est:Siotastar:integrale}
	\frac{1}{2\pi i}\oint_{|z|=1}\frac{\sG_\Omega(p,z)\d z}{(\log x)^{1-z}}=\sum_{1\leq j\leq \min\{J+1;\,\pi(p)\}}\frac{\res(\sG_\Omega(p,z);1/{\gp_j})}{(\log x)^{1-1/{\gp_j}}}+O_J\bigg(\oint_{|z|=\delta}\frac{|{\sG_\Omega}(p,z)|}{(\log x)^{1-\Re z}}\lvert{\rm d}z\rvert\bigg)\cdot
\end{equation}
Puisque, pour $p\leq\log x$ et $|z|=\delta$, nous avons la majoration
\[|{\sG_\Omega}(p,z)|\ll\frac{|F_\Omega(p,1/\delta)|}{|F_\Omega(p,z)|}
\ll(\log p)^{\delta+1/\delta},\]
il vient
\begin{equation}\label{eq:majo_integrale_erreur}
	\oint_{|z|=\delta}\frac{|{\sG_\Omega}(p,z)|}{(\log x)^{1-\Re z}}\lvert{\rm d}z\rvert\ll\frac{(\log p)^{\delta+1/\delta}}{\log x}\int_{0}^{2\pi}(\log x)^{\delta\cos\vartheta}\d\vartheta\ll\frac1{(\log x)^{1-1/{\gp_{J+1}}}}\cdot
\end{equation}
Les estimations \eqref{eq:est:S;Sstar:Omega}, \eqref{eq:est:Siotastar:integrale} et \eqref{eq:majo_integrale_erreur} impliquent directement, d'après \eqref{eq:est:Siotastar:aftera},
\begin{equation}\label{eq:SOmstardevAsRes}	
	\begin{aligned}
		S_{\Omega,\iota}(x)&=\sum_{p\leq\log x}\frac{x}{p^2}\bigg\{\sum_{1\leq j\leq\min\{J+1;\,\pi(p)\}}\frac{\res(\sG_\Omega(p,z);1/{\gp_j})}{(\log x)^{1-1/{\gp_j}}}+O_J\bigg(\frac{1}{(\log x)^{1-1/{\gp_{J+1}}}}\bigg)\bigg\}\cdot\\
		&=\sum_{1\leq j\leq J+1}\frac{x}{(\log x)^{1-1/{\gp_j}}}\sum_{{\gp_j}\leq p\leq \log x}\frac{\res(\sG_\Omega(p,z);1/{\gp_j})}{p^2}+O_J\bigg(\frac{x}{(\log x)^{1-1/{\gp_{J+1}}}}\bigg)\cdot
	\end{aligned}
\end{equation}
En outre, posant
\begin{equation*}\label{def:Spistar}
	S_{\Omega,\pi}^*(x):=\sum_{p\leq \log x}\frac1{2\pi pi}\oint_{|z|=1}\bigg(\sum_{\substack{P^+(A)\leq p\\ \Omega(A)\in\sK_{\Omega,x}}}z^{-\Omega(A)}\sum_{\substack{B\leq x/Ap\\ P^-(B)\geq p}}z^{\Omega(B)}\bigg)\frac{\d z}{z^2}\quad(x\geq 3),
\end{equation*}
il est clair que l'estimation \eqref{eq:est:S;Sstar:Omega} reste valable sous la forme
\begin{equation}\label{eq:est:Spi;Spistar}
	S_{\Omega,\pi}(x)=S_{\Omega,\pi}^*(x)+O\Big(\frac{x}{\log x}\Big)\quad(x\geq 3).
\end{equation}
Remarquant, de plus, que
\begin{equation}\label{eq:est:Spi:Siota}
	\res\Big(\frac1z\sG_\Omega(p,z);\frac1{\gp_j}\Big)=\gp_j\res\Big({\sG_\Omega(p,z)};\frac1{\gp_j}\Big),
\end{equation}
les estimations \eqref{somme_complete}, \eqref{eq:SOmstardevAsRes}, \eqref{eq:est:Spi;Spistar} et \eqref{eq:est:Spi:Siota} impliquent alors
\begin{equation}\label{eq:SOmstardevAsRes:2}
	S_{\Omega}(x)=\sum_{1\leq j\leq J+1}\frac{(1+\gp_j)x}{(\log x)^{1-1/{\gp_j}}}\sum_{p\geq{\gp_j}}\frac{\res(\sG_\Omega(p,z);1/{\gp_j})}{p^2}+O_J\bigg(\frac{x}{(\log x)^{1-1/{\gp_{J+1}}}}\bigg),
\end{equation}
avec $\res(\sG_\Omega(p,z);1/{\gp_j})\ll g_\Omega(p,1/{\gp_j})F_\Omega(p,{\gp_j})\ll(\log p)^{{\gp_j}-1/{\gp_j}}$, puisque 
\[\sum_{p>\log x}\frac{\res(\sG_\Omega(p,z);1/{\gp_j})}{p^2}\ll\frac{1}{\sqrt{\log x}}\sum_{p}\frac{(\log p)^{{\gp_j}-1/{\gp_j}}}{p^{3/2}}\ll\frac1{\sqrt{\log x}}\cdot\]
La formule \eqref{eq:somme_inverse:avec_mult} découle alors de \eqref{eq:SOmstardevAsRes:2} et de la définition des $\gc_{j}$ en \eqref{def:g;c}. Cela termine la démonstration du Théorème \ref{thm:somme_inverse:avec_mult}.\qed





\begin{thebibliography}{}
	\bibitem{bib:alladi}
	K. Alladi,
	The distribution of $\nu(n)$ in the sieve of Eratosthenes, \textit{Quart. J. Math. Oxford Ser.} (2) \textbf{33} (1982), no. 130, 129-148. MR657120.
	\bibitem{bib:debruijn}
	N.G. de Bruijn, 
	{\it Asymptotic methods in analysis}, North Holland (Amsterdam), 3e éd. ; réimpression : Dover (New York), 1981.
	\bibitem{bib:dekoninck:random}
	J.M. De Koninck and J. Galambos, Some randomly selected arithmetical sums, 
	{\it Acta Math. Hung.} \textbf{52} (1988), 37–43.
	\bibitem{bib:dekoninck:Pk}
	J.M. De Koninck, Sur les plus grands facteurs premiers d’un entier, 
	{\it Monatsh. Math.} \textbf{116} (1993), 13–37.
	\bibitem{bib:dekoninck}
	J.M. De Koninck \& F. Luca, On the middle prime factor of an integer, 
	{\it J. Integer Seq.,} 16 (5) :13.5.5, 2013.
	\bibitem{bib:bret}
	R. de la Bretèche \& G. Tenenbaum, 
	Entiers friables : inégalité de Tur\'an-Kubilius et applications, {\it Invent. Math.} \textbf{159} (2005), 531-588.
	\bibitem{bib:doyon}
	N. Doyon, V. Ouellet,
	On the sum of the reciprocals of the middle prime factor counting multiplicity, \textit{Annales Univ. Sci. Budapest., Sect. Comp.} \textbf{47} (2018) 249-259.
	\bibitem{bib:erdos_ivic}
	P. Erd\H os, A. Ivi\'c, and C. Pomerance, On sums involving reciprocals of the largest prime factor of an integer, 
	{\it Glas. Mat. Ser. III} \textbf{21} (1986), 283–300.
	\bibitem{bib:erdos}
	P. Erd\H os \& G. Tenenbaum, Sur les densités de certaines suites d'entiers,
	{\it Proc. Lond. Math. Soc.} (3) \textbf{59} (1989), 417-438.
	\bibitem{bib:pollack}
	N. McNew, P. Pollack {\&} A. Singha Roy,
	The distribution of intermediate prime factors, \textit{Illinois J. Math.} 68 (3) 537-576, September 2024.
	\bibitem{bib:ouellet}
	V. Ouellet, On the sum of the reciprocals of the middle prime factors of an integer, 
	{\it J. Integer Seq.} 20 (2017), no. 10, Art. 17.10.1, 16.
	\bibitem{bib:papier1}
	J. Rotgé, Étude statistique du facteur premier médian, 1: valeur moyenne, prépublication, arXiv:2501.15947.
	\bibitem{bib:papier2}
	J. Rotgé, Étude statistique du facteur premier médian, 2: lois locales, prépublication, arXiv:2501.15951.
	\bibitem{bib:papier3}
	J. Rotgé, Étude statistique du facteur premier médian, 3 : lois de répartition, \textit{Bull. Sci. Math.} \textbf{203} (2025) 103641.
	\bibitem{tenenbaum_livre}
	G. Tenenbaum,
	\it Introduction à la théorie analytique et probabiliste des nombres, \rm
	5e édition, Dunod, 2022.
\end{thebibliography}
\end{document}